\documentclass{amsart}
\usepackage{tikz}
\usetikzlibrary{calc,decorations.markings,decorations.pathmorphing,arrows}
\usetikzlibrary{shapes.misc}
\usepackage{amsmath, amssymb, amsthm, amsfonts, mathrsfs, amsfonts}
\usepackage{float}   
\usepackage[a4paper,left = 2.5cm, right= 2.5cm, top = 2.5cm, bottom = 2.5cm]{geometry}
\usepackage{xcolor}
\usepackage[all,tips]{xy}

\usepackage{tensor}

\usepackage[linktocpage, pagebackref, colorlinks=true, linkcolor=purple, urlcolor=violet, 
citecolor=blue, anchorcolor=black, hyperindex=true, hyperfigures]{hyperref}

\usepackage{graphicx}

\newcommand{\Z}{\ensuremath{\mathbb{Z}}}

\newcommand{\C}{\ensuremath{\mathbb{C}}}

\newtheorem{thm}{Theorem}[section]

\newtheorem*{question*}{Question}

\newtheorem{cor}[thm]{Corollary}
\newtheorem{lemma}[thm]{Lemma}
\newtheorem{prop}[thm]{Proposition}
\theoremstyle{remark}
\newtheorem{rmk}[thm]{Remark}

\theoremstyle{definition}
\newtheorem{defin}[thm]{Definition}
\newtheorem{notation}[thm]{Notation}
\newtheorem{algo}[thm]{Algorithm} 

\newtheorem{ex}[thm]{Example}

\newcommand{\OO}{\ensuremath{\mathcal{O}}}

\newcommand{\on}{\operatorname}
\newcommand{\tw}[2]{\tensor[^#1]{#2}{}}

\newcommand{\sym}{\rm sym}

\begin{document}
\title{Orbifold diagrams}


\author{Karin Baur}
  \address{School of Mathematics, University of Leeds, Leeds, LS2 9JT \\
  Currently on leave from the University of Graz, Graz, Austria.}
  \email{ka.baur@me.com}

\author{Andrea Pasquali} 
  \email{andrea.pasquali91@gmail.com}

\author{Diego Velasco}
  \email{vediegoez@gmail.com}

\maketitle

\begin{abstract}
	We study alternating strand diagrams on the disk with an orbifold point. These are quotients by rotation of Postnikov diagrams on the disk, and we call them orbifold diagrams. We associate a quiver with potential to each orbifold diagram, in such a way that its Jacobian algebra and the one associated to the covering Postnikov diagram are related by a skew-group algebra construction. 
	We moreover realise this Jacobian algebra as the endomorphism algebra of a certain explicit cluster-tilting object. This is similar to (and relies on) a result by Baur-King-Marsh for Postnikov diagrams on the disk.
\end{abstract}


\section{Introduction}

In this article we study \emph{orbifold diagrams}, i.e.~alternating strand diagrams on the disk with an orbifold point. These are collections of oriented arcs satisfying certain properties, which we define as quotients by rotation of alternating strand diagrams on the disk (also called Postnikov diagrams).
The latter have been used in the study of the coordinate ring of the Grassmannian: they give rise to clusters of the Grassmannian cluster algebras, \cite{Scott06}, or to cluster tilting 
objects of the Grassmannian cluster categories \cite{JKS16}, \cite{BKM16}. 
On the other hand, orbifolds have also been related to cluster structures, \cite{PS19}, \cite{CS14}. 
In~\cite{AP18}, Amiot and Plamondon construct
cluster algebras on surfaces with orbifold points of order 2, and in their construction 
skew-group algebras appear naturally.
Here we associate quivers with potentials to orbifold diagrams in such a way that skew-group algebras play a major role.

Skew group construction have been used in representation theory, for example in the seminal work of 
Reiten and Riedtmann~\cite{RR85} and of Asashiba~\cite{As11}. 
In \cite{LF-V}, the authors consider the triangulated disk with one orbifold point of order three. However in their set-up, 
the authors do not need skew group constructions because the action considered is free. The authors obtain a 
generalised cluster algebra from the Jacobian algebra associated to each triangulation of the aforementioned 
triangulated orbifold. Let us point out there is a well-known relation between triangulated surfaces and certain 
Postnikov diagram, see \cite[Section 13]{BKM16}, first described for the disk by Scott in her 
work~\cite[Section 3]{Scott06}. This relation allow us to expect a generalised 
cluster structure from the constructions we give in this paper. We will investigate this in future work. 

Our set-up is the following. 
We start with Postnikov diagrams with rotational invariance, i.e.~with an action of a cyclic 
group $G$ of order $d$, and take the quotient with respect to this action. 
We also give an intrinsic definition of 
such a quotient as a new combinatorial datum associated to a disk with an orbifold point, and call this 
an orbifold diagram.
We associate a quiver with potential $(Q_\OO,W_\OO)$ to every orbifold diagram $\mathcal O$, with a construction that 
depends on whether the orbifold point corresponds to a vertex of the quiver or not. In particular, we give a construction in case the action is not free on vertices.
In Proposition~\ref{prop:sgas} we prove that the frozen Jacobian algebras $A_{\mathcal O}$ 
of this new quiver and the one of the associated Postnikov diagram are related by a skew-group construction. 

We then restrict to the case where the permutation induced by the strands of the associated Postnikov diagram 
on the cover is of Grassmannian type $(k,n)$, to use results from~\cite{JKS16,BKM16}. 
As for Postnikov diagrams, there is an idempotent subalgebra $B({\mathcal O})$ 
of the frozen Jacobian algebra $A({\mathcal O})$ that only depends on $(k,n,d)$. 

Our aim is to realise the frozen Jacobian algebra as an endomorphism algebra of a cluster tilting object as in  
the statement in~\cite[Theorem 10.3]{BKM16}. 
To do this, we construct modules over the idempotent algebra $B(\OO)$ 
of an orbifold diagram 
in such a way that they are the images of the rank 1 modules from~\cite{JKS16} under a canonical functor.  

Any orbifold diagram determines a collection of such modules whose direct sum is a cluster-tilting objects in 
a Frobenius, stably 2-Calabi-Yau category. Our main result, Theorem~\ref{thm:main}, 
is that the endomorphism ring of this 
cluster-tilting object is isomorphic to $A(\OO)$.

\subsection*{Conventions}
We always consider finitely generated left modules and we compose arrows from right to left. 
The base field is the complex numbers. 

\subsection*{Acknowledgements}

Work on this paper started when the second and the third author visited the first author in Graz in 2018. We thank for the support provided by the Department of Mathematics of the University of Graz. All authors thank Ana Garcia Elsener and Matthew Pressland for helpful discussions. K.~B.~was supported by a Royal Society Wolfson Fellowship 180004 and by FWF grants P 30549 and W1230 and by the EPSRC Programme Grant W007509. She is currently on leave from the University of Graz. A.~P.~was supported by Uppsala University and the Alexander von Humboldt Foundation. D.~V.~was supported by the grant CONACyT-238754.

The authors thank the referees for their careful work and for their useful suggestions


\section{Orbifold diagrams}

In this section, we define orbifold diagrams on the disk with an orbifold point. Informally, these 
are quotients of rotation-invariant Postnikov diagrams, also called alternating strand diagrams. 
We start by defining these, following~\cite{Postnikov}.

We write $S_n$ for the symmetric group of permutations of $n$ elements. 

\begin{defin}\label{def:P-diagram}
	A \emph{Postnikov diagram of type $\sigma\in S_n$} is a collection of $n$ oriented curves $\gamma_i$, called \emph{strands}, on a disk with $n$ marked points on the boundary (clockwise labeled $1, \dots, n$), such that
	\begin{enumerate}
		\item The strand $\gamma_i$ connects the boundary point $i$ with $\sigma(i)$, starting at $i$. 
		The strand $\gamma_i$ intersects the boundary 
		only in those two (possibly coinciding) points. 
		\item There are a finite number of crossings, all between two strands, all transverse.
		\item Following a strand, the strands crossing it come alternatingly from the left and from the right. This includes strands crossing at boundary points. 
		\item[(4)] If two strands cross in two points $A$ and $B$, then one is oriented from $A$ to $B$ and the other is oriented from $B$ to $A$. This also applies to crossings at boundary points.
		\item[(5)] If a strand crosses itself other than at a boundary point, then consider the disk determined by the loop. No strand intersects the interior of this disk.
	\end{enumerate}
	A \emph{Grassmannian Postnikov diagram of type $(k,n)$} is a Postnikov diagram satisfying the additional condition
	\begin{enumerate}
		\item[(6)] The permutation $\sigma\in S_n$ is given by $\sigma(i) = i + k\pmod {n}$.
	\end{enumerate}
	Postnikov diagrams are considered modulo isotopy fixing the boundary.
\end{defin}

\begin{rmk}\label{rem:reduce}
	Postnikov diagrams can be reduced as follows, see Figure~\ref{fig:PullingStrand}: \\
	(i) If two strands cross in points $A$ and $B$ such that the region formed by $A$ and $B$ is simply 
	connected then we can 
	reduce by ``pulling the strands'' in a way to remove the two crossings. 
	Note that one of the points $A$ and $B$ may be a marked point on the boundary; in that case, only one crossing gets removed. 
	
	\noindent 
	(ii) If a strand crosses itself and if the disk determined by the loop contains no other strands, the strand 
	can be straightened, i.e. the crossing removed. 
	
	Diagrams reduced in this way retain many properties of the original diagram, and so we will often assume in the following that Postnikov diagrams are reduced.
\end{rmk}

\begin{figure} 
	\includegraphics[scale=.4]{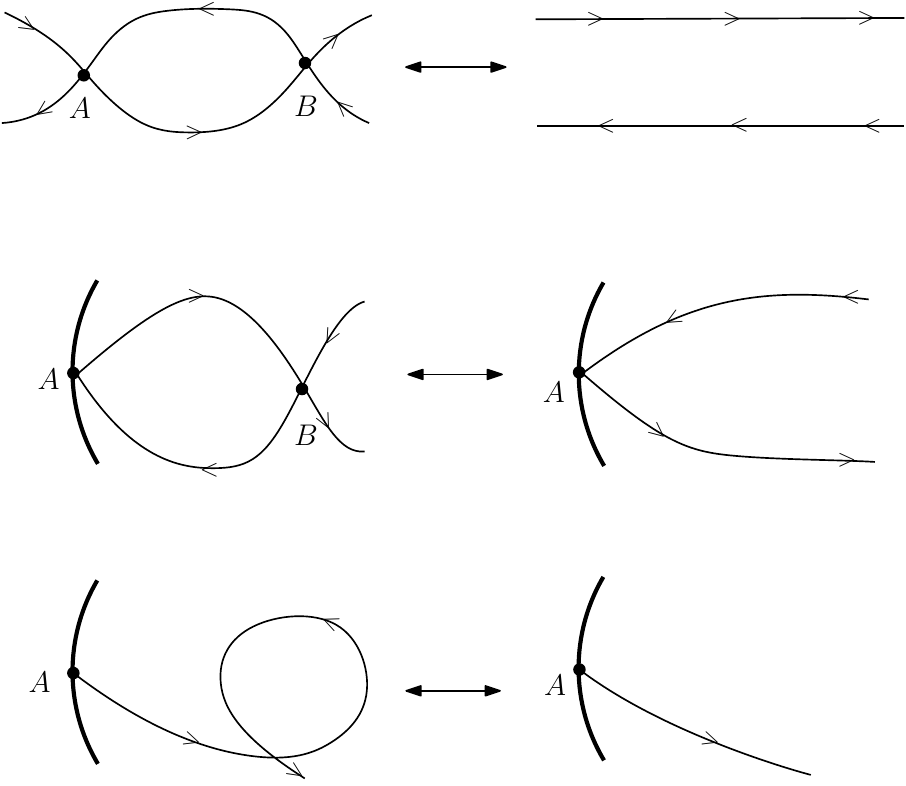}
	\caption{Pulling strands in order to reduce a Postnikov diagram.}
	\label{fig:PullingStrand}
\end{figure}

Since we plan to take quotients by rotations of the disk, an important role is played by the 
Postnikov diagrams which are rotation-invariant. These were first studied in~\cite{Pasquali20} in relation to self-injective Jacobian algebras.

\begin{defin}
	A Postnikov diagram of type $\sigma \in S_n$ is \emph{$d$-symmetric} if it is (up to isotopy) invariant under rotation by $\frac{2\pi}{d}$.
\end{defin}

Observe that in this case $\Z_d$ must act freely on $\{1, \dots, n\}$, and so $d\mid n$.

\begin{ex}\label{ex:P-examples}
	Figure~\ref{fig:P-examples} shows examples of Postnikov diagrams. The first is of type $\sigma= (13764)(258)\in S_8$, the second is a symmetric 
	Postnikov diagram of type $\sigma=(194276)(258) \in S_9$ and the last is a symmetric 
	Grassmannian Postnikov diagram of type $(4,10)$. 
	\begin{figure}
		\includegraphics[scale=.48]{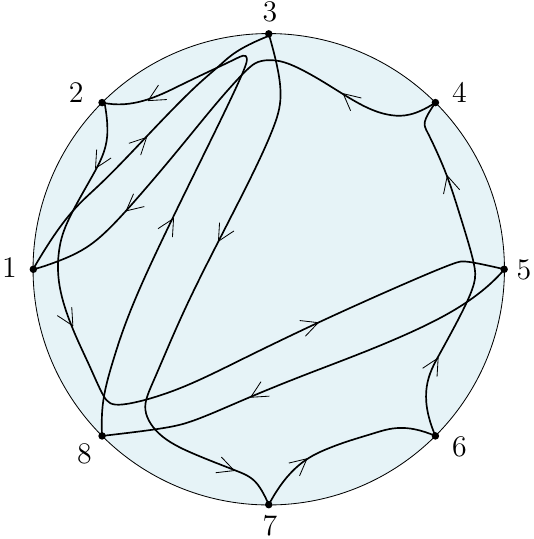}
		\hskip .5cm
		\includegraphics[scale=.58]{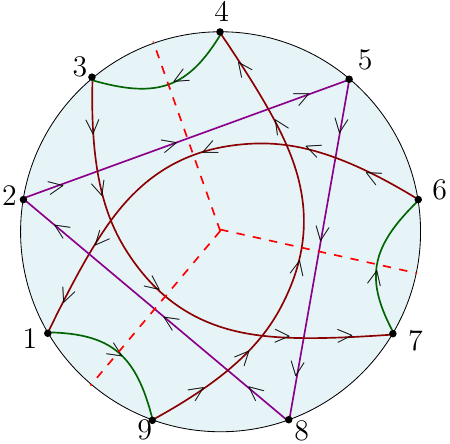}
		\hskip .5cm
		\includegraphics[scale=.38]{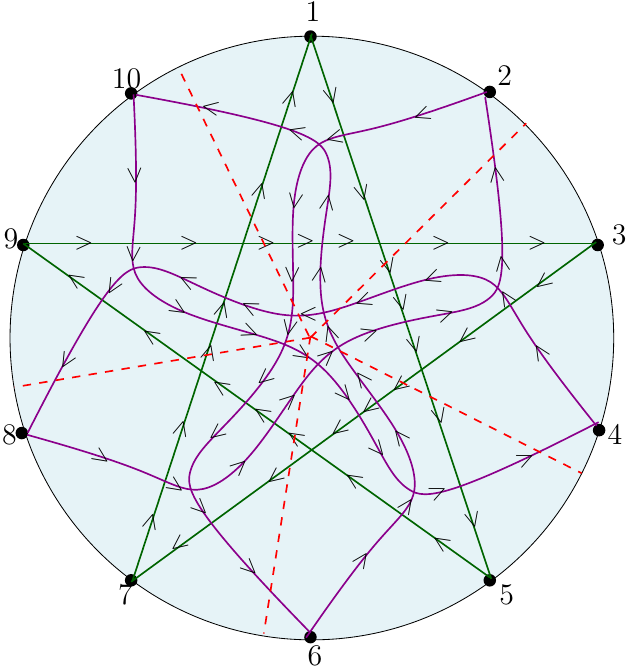}
		\caption{Examples of Postnikov diagrams. Symmetry axes indicated by dashed lines.}
		\label{fig:P-examples}
	\end{figure}
\end{ex}

If we start with a $d$-symmetric Postnikov diagram of order $d>1$, 
we can construct its (topological) quotient by the cyclic group of order $d$ acting by rotations. This will be a collection of curves on a disk with an orbifold 
point of order $d$. The resulting diagram is what we will call an ``orbifold diagram''. 

We first give an abstract definition of a (weak) orbifold diagram and introduce orbifold diagrams in 
Definition~\ref{def:DefOrbiDiagram}. 
We will then show that orbifold diagrams as defined through this are the same 
as quotients of symmetric Postnikov diagrams 
(Proposition~\ref{prop:orbifold-is-quotient}). 

\begin{notation}
	We will use the usual notion of winding number for a closed curve with respect to a point, but the clockwise direction is for us positive. This is because in the literature the boundary points are usually labeled clockwise.
\end{notation}

\medskip 
Let $\Sigma$ be a disk with $n_0$ marked points on the boundary (clockwise labeled $1,\dots, n_0$) 
and an orbifold point $\Omega$ of order $d>1$. 

\begin{defin}\label{def:weak-orbifold}
	A \emph{weak orbifold diagram of type $\tau\in S_{n_0}$ on $\Sigma$} 
	is a collection of $n_0$ 
	oriented curves $\gamma_{ {i}}$, called {\em strands}, on $\Sigma$, such that
	
	\begin{enumerate}
		\item The strand $\gamma_{ {i}}$ connects the boundary point $ {i}$ with $\tau( {i})$, starting from $i$.
		The strand $\gamma_{ {i}}$ intersects the boundary 
		only in those two (possibly coinciding) points, and does not go through $\Omega$. 
		\item There is a finite number of crossings, all between two strands, all transverse.
		\item Following a strand, the strands crossing it come alternatingly from the left and from the right. This includes strands crossing at boundary points. 
		\item\label{DefWO4} If two strands cross in two points $A$ and $B$ and both are oriented from $A$ to $B$, then consider the closed curved formed by following a strand from $A$ to $B$ and then following the other strand in the opposite direction from $B$ to $A$. The winding number of this closed curve with respect to $\Omega$ is not 0 (for an example, 
		see the curves between the points $Q_1$ and $Q_2$ in both pictures in 
		Figure~\ref{fig:example1New}). 
		\item If a strand crosses itself, then consider the closed curve formed by following the strand from a point of intersection to itself. 
		Either this has nonzero winding number with respect to $\Omega$, or it is a simple loop not intersecting any other strand (and thus can be reduced as for Postnikov diagrams).
	\end{enumerate}
\end{defin}
Weak orbifold diagram are considered up to isotopy fixing the boundary and the center of the disk.
Weak orbifold diagrams can be reduced like Postnikov diagrams, provided that strands do not need to be moved 
across the orbifold point when doing so.

\begin{figure}
	\includegraphics[scale=.38]{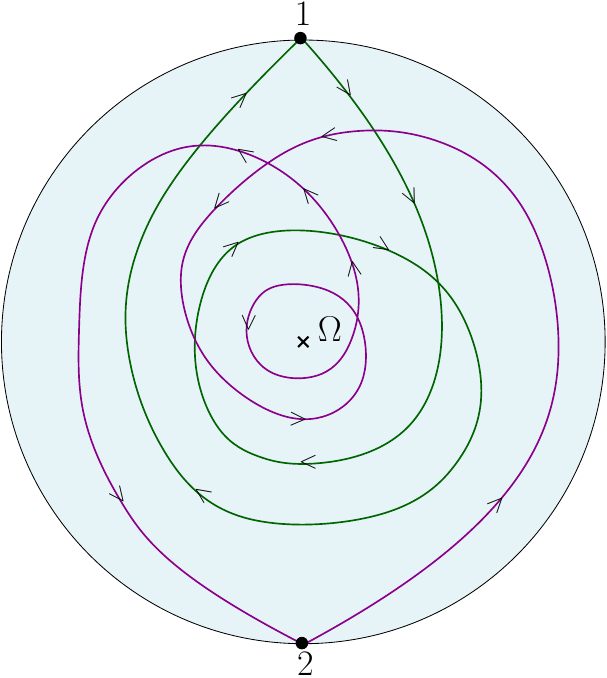}
	\hskip 1cm
	\includegraphics[scale=.6]{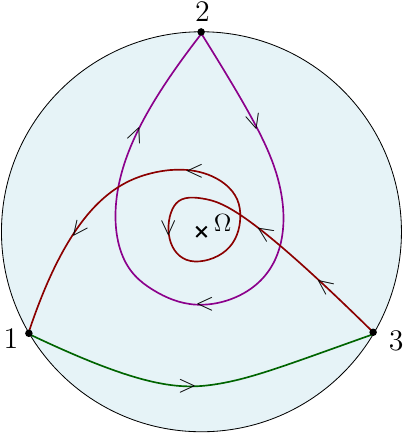}
	\caption{Two weak orbifold diagrams; $\tau=\on{id}\in S_2$ on the left, $\tau=(13)\in S_3$ on the right.}
	\label{fig:weak-orbifold}
\end{figure}

Figure~\ref{fig:weak-orbifold} shows examples of weak orbifold diagrams. 
Observe that the order $d$ of the orbifold point $\Omega$ 
does not appear in the axioms: it is part of the datum of the surface. 
So we can have the same diagram (picture) for varying orders $d$.

\begin{rmk}
	To any weak orbifold diagram $\OO$ on a disk $\Sigma$ we will consider a symmetrized 
	version of $\OO$ on the universal cover of $\Sigma$. This depends on the order of $\Omega$, 
	in particular, the same strand configuration (picture) leads to a symmetrized version for every $d>1$. 
\end{rmk}

\begin{defin}
	Let $\OO$ be a (reduced) weak orbifold diagram on $\Sigma$, assume that $\Omega$ has order $d$. 
	Draw a simple curve joining $\Omega$ to the boundary arc between $n_0$ and 1. 
	Then 
	${\sym}_d(\OO)$ is the collection of $n_0d$ strands obtained from taking $d$ copies of $\OO$ and 
	gluing them along the copies of the simple curve. We draw the resulting surface as a disk and 
	label the marked points by 
	$1,2,\dots, n_0,n_0+1,\dots, dn_0$ clockwise around the boundary. 
\end{defin}

By construction, the image ${\sym}_d(\OO)$ is a collection of $n= n_0d$ strands on a disk 
(without orbifold points)
which is symmetric under rotation by $\frac{2\pi }{d}$. The image 
${\sym}_d(\OO)$ corresponds to taking the universal cover of the orbifold diagram $\OO$ for the 
surface $\Sigma$ with $\Omega$ a point of order $d$. 

The result is not a Postnikov diagram in general, as it may 
have ``lenses" (pairs of twice-crossing parallel strands) and self-crossings (compare 
Definition~\ref{def:weak-orbifold} and Defintion~\ref{def:P-diagram}). 
This is illustrated in Example~\ref{Exam2} 
below. If $d$ is large enough, the symmetrized version ${\sym}_d(\OO)$ of $\OO$ 
is a Postnikov diagram, which is $d$-symmetric by 
construction, see Proposition~\ref{prop:orbifold-is-quotient}. 

Note that the 
the quotient of ${\sym}_d(\OO)$ under the rotation by $\frac{2\pi }{d}$ is $\OO$. 
We will write $\mathcal P/d$ to denote the quotient of a $d$-symmetric Postnikov diagram under the 
rotation by $\frac{2\pi }{d}$. 

\begin{figure}[H]
\centering
	\includegraphics[scale=.6]{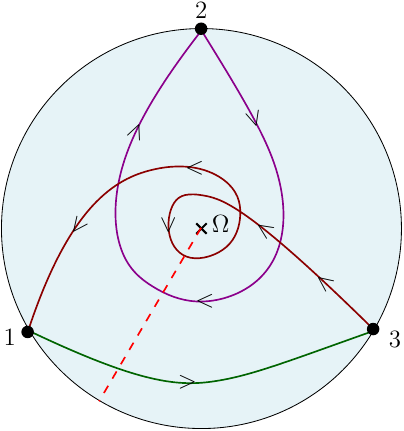}
	\hskip .5cm
	\includegraphics[scale=.6]{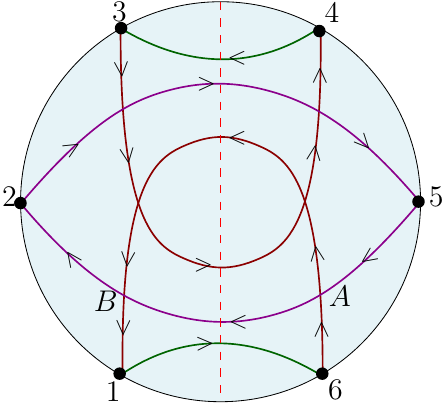}
	\hskip .5cm
	\includegraphics[scale=.6]{figs/ex-n0-3-unfold-dv-O-dv.pdf}
	\caption{A weak orbifold diagram on $\Sigma$, its 2-fold cover in the middle and its 3-fold cover on the right. 
		Red dashed lines indicate the fundamental domains/symmetry axes.}
	\label{fig:example2}
\end{figure}

\begin{ex}\label{Exam2}
	Here we start with a weak orbifold diagram $\OO$ for $\tau=(13)\in S_3$ on $\Sigma$, 
	with orbifold point $\Omega$ of order $d$, see first picture in Figure~\ref{fig:example2}. 
	We consider ${\sym}_d(\OO)$ for $d=2$ and $d=3$. 
	
	Let us consider the 2-fold cover in Figure~\ref{fig:example2}. This is not a good cover of $\OO$ for two independent reasons. First, it is not a Postnikov diagram, since it violates condition (4) of 
	Definition~\ref{def:P-diagram}: the strands crossing at $A$ and $B$ are both oriented from $A$ 
	to $B$. 
	This is because the order of the orbifold point (i.e.~2) is too small compared to how much the strands wind around it. In Definition~\ref{def:DefOrbiDiagram} we will precisely quantify how large $d$ needs to be for the $d$-fold cover to be a Postnikov diagram.
	
	The second issue is more subtle: the diagram of the 2-fold cover is not reduced, in the sense that we can apply a reduction move as in Remark~\ref{rem:reduce}. However, the quotient of the reduced diagram by the rotation of order 2 is not the same as $\OO$ (it corresponds to applying a forbidden reduction move that goes through $\Omega$). 
	This issue arises because the order of the orbifold point is exactly 2. Indeed, the reduction moves are applied to digons, and those arise precisely from covers of order 2. To avoid this, we will stipulate that the order of orbifold diagrams is at least 3, which ensures that if $\OO$ is reduced then its cover is also reduced.
	
	Finally, the 3-fold cover ${\sym}_3(\OO)$ is a 3-symmetric Postnikov diagram: the problems disappear, since 3 is large enough (as per Definition~\ref{def:DefOrbiDiagram}) and is not equal to 2.
\end{ex} 

Let us point out that if we start from a $d$-symmetric Postnikov diagram on a disk with 
$n=n_0d$ marked points and take its quotient under the rotation by $\frac{2\pi}{d}$, we obtain 
a weak orbifold diagram on a disk $\Sigma$ with $n_0$ points with additional properties, see Example~\ref{Exam1}.

\begin{ex}\label{Exam1}
	We start with a 5-symmetric Grassmannian Postnikov diagram of type $(4,10)$, 
	see Figure \ref{fig:example1}. 
	When we quotient by the 5-fold symmetry we get a weak
	orbifold diagram on a disk $\Sigma$ with a point $\Omega$ of order 5 and 
	with $2$ marked points. The type of the image is $\tau=\on{id}$. 
	\begin{figure}[H]
		\includegraphics[scale=.55]{figs/example-4-10-dv-O-dv.pdf}
		\hskip 1cm
		\includegraphics[scale=.55]{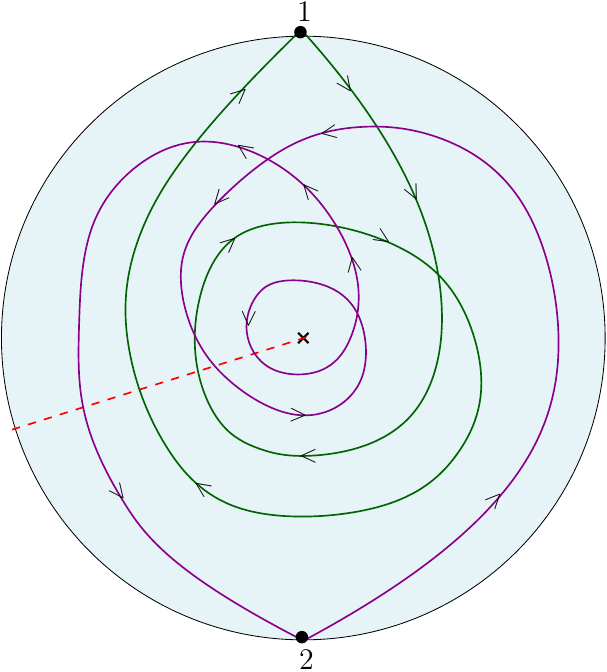}
		\caption{Taking the quotient under rotation by $2\pi/5$ gives a weak orbifold with a 
		point of order 5; 
			red dashed lines indicate symmetry axes/fundamental domains.}
		\label{fig:example1}
	\end{figure}
\end{ex}

\begin{ex}\label{ex:middle-alternating}
Here we have a 3-symmetric Postnikov diagram of type $(3,9)$. 
Its quotient by the 3-fold symmetry, on the right, is a weak orbifold diagram on a disk $\Sigma$ with 
$\Omega$ of order 3 and 3 marked points, of type $\tau=\on{id}$. 
\begin{figure}[H]
	\includegraphics[scale=.51]{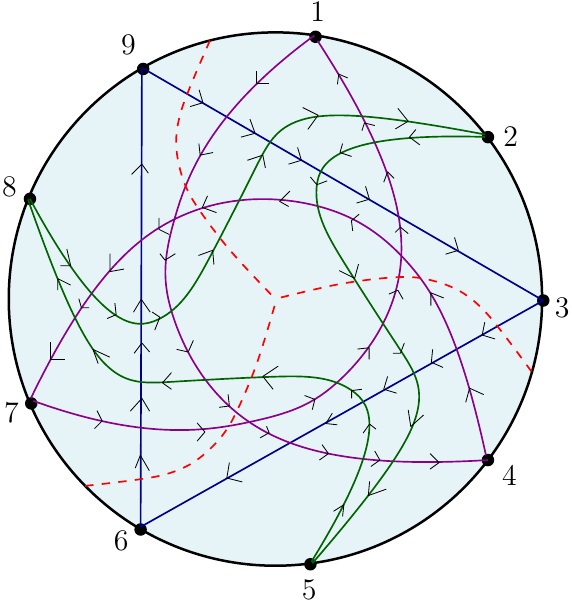}
	\hskip .3cm
	\includegraphics[scale=.54]{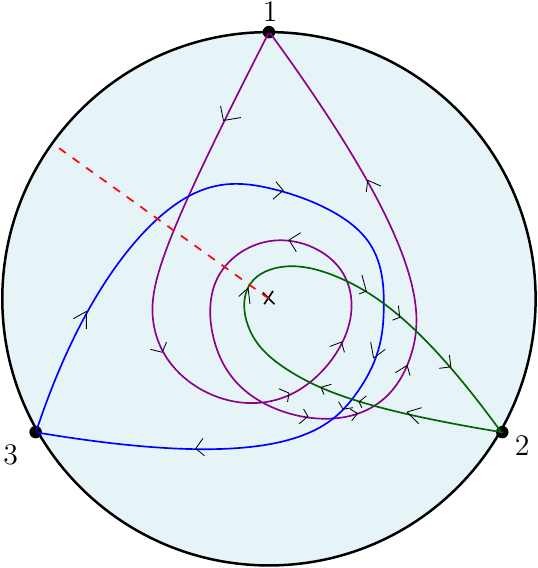}  
	\caption{A symmetric orbifold diagram $\mathcal P$ with its quotient $\mathcal P/3$ on the right.} 
	\label{fig:middle-alternating}
\end{figure}	

\end{ex}

We would like to upgrade our definition of weak orbifold diagram by including the value of $d$ in the datum of the picture, as well as guaranteeing that the $d$-fold cover is a Postnikov diagram. The only properties that might fail are $(4)$ and $(5)$ in Definition~\ref{def:P-diagram}. Since for sufficiently large $d$ these properties hold, 
we pick the smallest such $d$.

Let us define some notation. For a strand $\gamma$ in a weak orbifold diagram, consider its points of self-intersection (including at the boundary). Each of these points $P$ determines a closed subcurve of $\gamma$ (going from $P$ to itself), which has a winding number $w(P)$ with respect to $\Omega$. 
We define $S(\gamma)$ to be the maximum of the absolute values of $w(P)$, where $P$ varies in the set of self-intersections of $\gamma$. If $\gamma$ does not intersect itself we set $S(\gamma) = 0$.

Similarly, let $\gamma_1$ and $\gamma_2$ be two strands in an orbifold diagram. Assume that they meet in two points 
$A$ and $B$, and that they are both oriented from $A$ to $B$. 
Then consider the curve formed by following $\gamma_1$ from $A$ to $B$ and then $\gamma_2$ against the orientation from $B$ to $A$. 
This is a closed loop and it has a winding number $w(A,B)$ with respect to $\Omega$. 
Strictly speaking, this is not well-defined as the sign of $w(A,B)$ depends on the choice 
of the curve that is taken against the orientation. But we are only interested in the absolute value of 
$w(A,B)$: 
We define $L(\gamma_1, \gamma_2)$ to be the maximum of the absolute values of $w(A, B)$ for all pairs $A,B$ as above. 
We set $L(\gamma_1, \gamma_2) = 0$ if $\gamma_1$ and $\gamma_2$ do not meet as above. 

\begin{defin}\label{def:DefOrbiDiagram} 
	Let $\Sigma$ be a disk with an orbifold point $\Omega$ of order $d>1$. 
	A weak orbifold diagram $\mathcal O$ on $\Sigma$ 
	is an \emph{orbifold diagram (of order $d$)} if 
	\[
	d> \max\{\max_{\gamma}S(\gamma), \max_{\gamma_1\neq \gamma_2}L(\gamma_1,\gamma_2)\}.
	\]
	An orbifold diagram on $\Sigma$ is \emph{Grassmannian} if  $\tau=\on{id} $
	and there is an integer $0<w_+<d$ such that every strand has winding number $w_+$ or $w_+-d$.
	 In this case, we say that the orbifold diagram is \emph{of type} $(k,n)$, where $n= n_0d$ and $k = n_0w_+$.
\end{defin}

\begin{ex}\label{ex:ex2-ex1-orb}
We consider the weak orbifold diagrams from Examples \ref{Exam2} and \ref{Exam1}.
	
\begin{enumerate}
\item We first take the weak orbifold diagram $\OO$ on the left of 
Figure \ref{fig:example1New}. We want to see whether the conditions of 
Definition \ref{def:DefOrbiDiagram} hold. The strand $\gamma_1$ does not have self-intersection points, 
$w(P_2)=1$ and $w(P_3)=-1$, so $S(\gamma_1)=0$, $S(\gamma_2)=S(\gamma_3)=1$. For the 
second condition: we have $w(Q_1,Q_2)=2$ and so $L(\gamma_2, \gamma_3)=2$. 
Since $d=3$, $\OO$ is indeed an orbifold diagram.

\item 
Now we look at the diagram $\OO$ on the right of Figure \ref{fig:example1New}. 
This is a weak orbifold diagram of order $d=5$ and we want to see whether the conditions of 
Definition \ref{def:DefOrbiDiagram} hold. 
We get $w(P_{11})=2$, $w(P_{12})=1$ $w(P_{21})=-3$, $w(P_{22})=-2$ and $w(P_{23})=-1$. 
So $S(\gamma_1)=2$ and $S(\gamma_2)=3$. We have  $w(Q_1,Q_2)=1$, $w(Q_1, Q_3)=3$, $w(Q_1,Q_4)=4$, $w(Q_2,Q_3)=2$, $w(Q_2,Q_4)=3$ and $w(Q_3, Q_4)=1$. 
In this case $L(\gamma_1, \gamma_2)=4$. Since $d=5$, $\OO$ is also an orbifold diagram.

 \begin{figure}[H]
		\includegraphics[scale=.75]{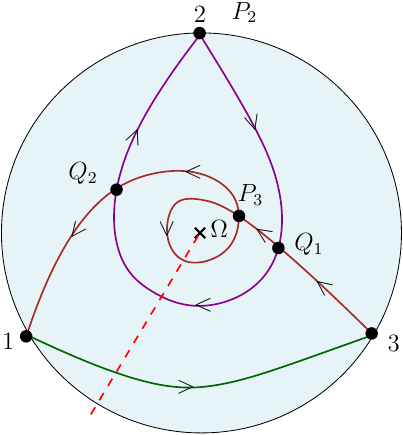}
		\hspace{1cm}
		\includegraphics[scale=.48]{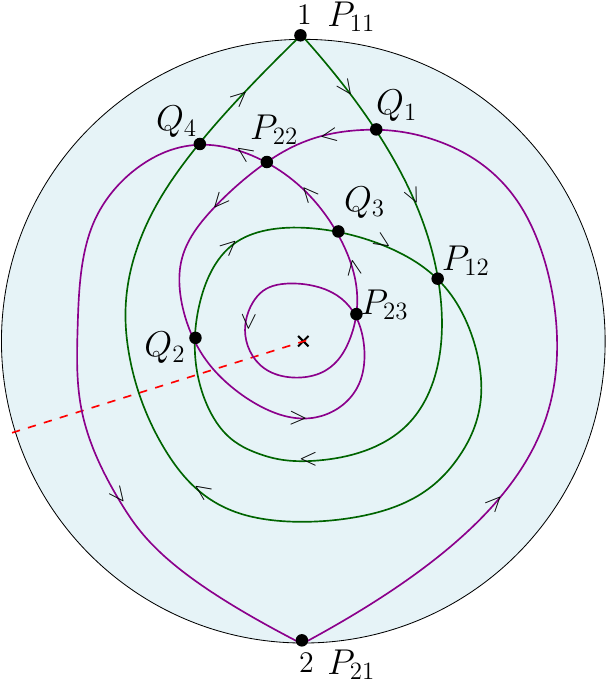}
		\caption{Computing the values of the winding numbers, 
		see Definition~\ref{def:DefOrbiDiagram}. }
		\label{fig:example1New}
	\end{figure} 	
\end{enumerate} 	
\end{ex}

\begin{prop}\label{prop:orbifold-is-quotient}
	Let $\OO$ be an orbifold diagram of order $d>2$, and let $\mathcal P$ be an 
	$s$-symmetric Postnikov diagram (for some $s>1$). 
	Then:
	\begin{enumerate}
		\item ${\sym}_d(\OO) $ is a $d$-symmetric Postnikov diagram.
		\item $\mathcal P/s$ is an orbifold diagram on a disk with an orbifold point 
		$\Omega$ of order $s$.
		\item ${\sym}_d(\OO)/d = \OO$ and, if $s>2$, ${\sym}_s(\mathcal P/s) = \mathcal P$.
	\end{enumerate}
\end{prop}

\begin{proof}
	Let us begin with (1). First, at every boundary point of ${\sym}_d(\OO)$ there is exactly one outgoing and one incoming strand, as this is true for $\OO$. Moreover, since every point on a strand can be reached by walking along a strand on $\OO$, the same is true for ${\sym}_d(\OO)$, which means that every piece of curve in ${\sym}_d(\OO)$ is indeed part of a strand coming from the boundary (i.e.~there are no closed strands inside the interior of the disk). So condition (1) of Definition~\ref{def:P-diagram} is satisfied. 
	Condition (2) holds by construction, as does condition (3) since it is local.
	Let us examine condition (5). Every self-crossing of a strand $\gamma$ on ${\sym}_d(\OO)$ comes from a self-crossing of a strand $\eta$ of $\OO$. Let $\eta'$ be the subcurve of $\eta$ 
	defined by such a crossing, i.e.~we start from a crossing point $Q$ and follow $\eta$ until 
	we reach $Q$ again. 
	Thus $\eta'$ is a closed curve in $\OO$. 
	By condition (5) of Definition~\ref{def:weak-orbifold}, this either can be reduced or has nonzero winding number. If it can be reduced, so can its images in ${\sym}_d(\OO)$ and condition (5) is satisfied. 
So assume that $\eta'$ cannot be reduced. Then by Definition~\ref{def:DefOrbiDiagram}, 
the absolute value of the winding number of $\eta'$ is strictly less than $d$. 
Without loss of generality, this winding number $w$ is positive, so $1<w<d$. 

Call $c$ the curve connecting $\Omega$ to the boundary of the disk in $\OO$ which we chose to construct ${\sym}_d(\OO)$. 
We may choose $c$ in a way to minimise the crossings with $\eta'$. 
Then $\eta'$ crosses $c$ exactly $w$ times. 

So if we follow the image of $\eta'$ in ${\sym}_d(\OO)$ starting from a chosen lift of the crossing 
point $Q$, it will reach another lift of $Q$ in the $w$-th copy of the fundamental domain, counting 
clockwise from the copy where $Q$ is. 
Since $w<d$, these two lifts are in different regions
so the lifts of the starting segment and of the ending segment of $\eta'$ 
belong to different strands in ${\sym}_d(\OO)$. 
In particular, the lifts of $\eta'$ do not violate condition (5) of Definition~\ref{def:P-diagram}. 

	An analogous argument applied to the closed subcurve associated to two strands crossing in 
two points (as in condition (4) of Definition~\ref{def:weak-orbifold}) shows that condition 
(4) must hold as well. 
We conclude that ${\sym}_d(\OO)$ is a Postnikov diagram, which is also invariant under rotation by 
$\frac{2\pi}{d}$ by construction.
	
	Now to prove claim (2). First, conditions (1)--(5) in Definition~\ref{def:weak-orbifold} follow 
each from the corresponding condition in Definition~\ref{def:P-diagram}. 
The inequality of Definition~\ref{def:DefOrbiDiagram} follows from the (converse of) the argument 
we used for claim (1): the points in $\mathcal P$ mapping down to a self-intersection in 
$\mathcal P/s$ must be distinct since $\mathcal P$ is a Postnikov diagram, and so 
the order $s$ is large enough. The same holds for two strands crossing in two points, 
and so $\mathcal P/s$ is an orbifold diagram.
	
Claim (3) is clear by definition of the operations ${\sym}_d $ and $/s$.
\end{proof}

%
\section{Labels on orbifold diagrams}\label{sec:labels}

	We will now explain how to associate equivalence classes of subsets of $\{1, \dots, n\}$ to alternating regions of an orbifold diagram, in a way that corresponds to the construction for Postnikov diagrams from~\cite{Postnikov}.
	
	Let $\OO$ be an orbifold diagram of order $d$ and of type $\tau \in S_{n_0}$.
	To it we associate the $d$-symmetric Postnikov diagram ${\sym}_d(\OO)$ as explained before. 
The latter has $n= n_0d$ marked points on the boundary and has type $\sigma$ for some $\sigma$ 
depending on $\tau$ and $\OO$.
	For Postnikov diagrams, the rule in~\cite{Postnikov} can be used to assign a $k$-element subset of 
$\{1, \dots, n\}$ to certain regions delimited by the strands (for a certain $k$ depending on $\sigma$). 
We recall this construction here. 	
	
	First observe that the complement of the strands of ${\sym}_d(\OO)$ in the disk is a disjoint union of 
topological disks, which can each be of one of three kinds. There are \emph{boundary regions}, whose boundary 
contains a segment (of positive length) of the boundary of the disk, 
and there are \emph{cyclical} and \emph{alternating} regions, depending on whether the strands
adjacent to them give their boundary a cyclic orientation or not. The strands adjacent to the boundary regions are alternatingly oriented and so we count these regions as alternating 
regions.
We will assign to each alternating region a label, which is a subset of $\{1, \dots, n\}$. 
We do this as follows: every strand divides the disk into two pieces, one on its left and one on its 
right (when following the strand in its orientation). A number $i$ is part of the label of an alternating region if 
and only if the region is in the left piece determined by the strand starting at vertex $i$. 
This procedure assigns a subset of some constant cardinality to every boundary and 
every alternating region as when we move from one alternating region to a neighbouring alternating region we always 
exchange one label for another one. 
If the Postnikov diagram is Grassmannian of type $(k,n)$, then this cardinality 
is equal to $k$. 
Examples of labels on (symmetric) Postnikov diagrams are on the left hand side 
of Figure~\ref{fig:NewLabelEx-1} and on the left of Figure~\ref{fig:dimer-relations}.
	
	If the Postnikov diagram is $d$-symmetric, then the labels of two regions related by rotation 
by $\frac{2\pi}{d}$ differ by adding $n_0 = n/d$ (addition on sets is meant pointwise). 
We use the labels of ${\sym}_d(\OO)$ to associate labels to the alternating regions of $\OO$ 
by taking equivalence classes of sets of labels under adding $n_0$ pointwise (that is, $\{i_1,\dots, i_k\}\sim_{n_0} \{h_1, \dots, h_k\}$ if there is $j$ such that $ \{i_1+ jn_0,\dots, i_k +jn_0\} = \{h_1, \dots, h_k\}).$ 
We use square brackets to denote the equivalence classes of sets of labels: 
$[i_1,\dots, i_k]_{n_0}$ for the set 
$\{\{i_1+jn_0, i_2+jn_0,\dots, i_k+jn_0\}\mid 1\le j\le d\}$ where all labels $i_l$ are from  
$\{1,2,\dots, n_0d\}$.
Every alternating region of $\OO$ corresponds to $d$ different alternating regions of ${\sym}_d(\OO)$ in general 
(see  Remark~\ref{rmk:central_region}) and as such to an equivalence class $[i_1,\dots, i_k]_{n_0}$ of labels. 
We assign this equivalence class to the alternating region, and do this for all alternating regions of $\OO$.

\begin{defin}
	\label{defin:I_O}
	Let $\OO$ be an orbifold diagram with $n_0$ boundary points. Let $\mathcal I$ be the collection of 
	labels of alternating regions of ${\sym}_d(\OO)$ and let $\sim_{n_0}$ be the equivalence relation on 
	$\mathcal I$ described above. We define $\mathcal I_\OO = \mathcal I/\sim_{n_0}$.
	By the previous discussion, $\mathcal I_\OO$ is the set of labels attached to the alternating regions of $\OO$. 
\end{defin}

\begin{rmk}\label{rmk:central_region}
	The equivalence classes $[i_1,\dots, i_k]_{n_0}$ usually contain $d$ elements, corresponding to the $d$ different regions of ${\sym}_\OO$ mapping down to a given region of $\OO$. A possible exception is the central region of $\OO$, in case it happens to be alternating: 
	its label is a single subset which is invariant under adding $n_0$ to its elements. In this case we have 
	$[i_1,\dots, i_k]_{n_0} = \{\{i_1,\dots, i_k\}\}$.
\end{rmk}

We now give a way to obtain the labels directly from the orbifold diagram, without going 
through the 
associated symmetric Postnikov diagram. We illustrate this algorithm in 
Examples~\ref{ex:labels-d-5}, 
\ref{ex:labels-d-3} and \ref{Ex:CentralFixedPoint}.  

\begin{algo}\label{alg:labels}
Step 1: Let $\mathcal O$ be an orbifold diagram of order $d$ on a disk with $n_0$ marked points. 
Let $n=dn_0$. 
Draw a curve $\gamma$ from the orbifold point 
$\Omega$ to the boundary of the disk which ends between $n_0$ and $1$ 
(see Remark~\ref{rem:orbifold-labeling} (1)) such that $\gamma$ crosses the strands 
$\gamma_i$ transversally and never goes through a crossing of two strands. 

Step 2: 
The curve $\gamma$ divides (some of) the strands into different connected components which we call 
{\em segments}. 
We now label these different segments as follows. The strand $\gamma_i$ gets the label $i$ from its 
starting point to the first intersection with $\gamma$. 
If $\gamma_i$ leaves $i$ clockwise (i.e. when leaving $i$, it appears to follow the boundary in a clockwise way and 
the orbifold point is to the right of $\gamma_i$ when it crosses $\gamma$), we subtract $n_0$ from 
the label, reducing integers modulo $n$. If $\gamma_i$ leaves $i$ counterclockwise, we add $n_0$ to the label, 
reducing modulo $n$. The segment between the first crossing and the second crossing is then 
$i\mp n_0$ accordingly. We iterate this until all segments of each $\gamma_i$ are labeled. 
The labels on the segments of $\gamma_i$ are in $\{1,2,\dots, n\}$ since we reduce modulo $n$. 

Step 3: 
Every strand divides the surface into two regions, one on its left and one of its right (when following 
the strand in its orientation). Furthermore, the complement of all strands is a union of faces, one of them 
containing the orbifold point, where the boundary of each face is formed by parts of the strands and where each 
face is either cyclical or alternating. To every alternating region 
which is not incident with the curve 
$\gamma$, we associate the label $i+mn_0$ if the alternating region is to the left of the strand segment with 
label $i+mn_0$ (for some $m\in \mathbb{Z}$). 

Step 4:
Observe that the alternating regions through which $\gamma$ goes are {\em cut in two} if we open the disc 
$\Sigma$ along $\gamma$. We only associate labels to the region which is counterclockwise from 
$\gamma$ (see Remark~\ref{rem:orbifold-labeling}(2) below) as in Step 3: 
such an alternating region gets label $i+mn_0$ if it 
is to the left of the strand segment with label $i+mn_0$ (for some $m\in \mathbb{Z}$).

Step 5: ``Add missing labels'': After steps 3 and 4, every alternating region has a certain number of labels. This 
number is constant as whenever we go from one alternating region to a neighbouring alternating region, 
we cross exactly two strands, one in each direction, so one label gets added and one removed, keeping 
the number of labels constant. 
However, certain elements of $\{1,2,\dots, n\}$ do not appear as segment labels (Step 2). 
Let $j$ be such a label and let $j_0$ be its reduction 
modulo $n_0$. If the orbifold point $\Omega$ is to the left  of strand $\gamma_{j_0}$, 
we associate the label $j$ to 
every alternating region of $\mathcal O$. 
If not, the label $j$ does not appear in any of the regions. 
\end{algo}

\begin{rmk}\label{rem:orbifold-labeling} 
(1) The curve $\gamma$ breaks $\mathcal O$ open so that it can be viewed to be a copy of the 
fundamental domain of ${\sym}_d(\mathcal O)$, with marked points $i,i+1,\dots, i+n_0-1$ along the boundary 
(for some $i\in \{1,\dots, dn_0\}$). It is important that $\gamma$ links the orbifold point 
$\Omega$ with the boundary segment between $n_0$ and $1$, in order for the algorithm to agree with Definition~\ref{defin:I_O} without further adjustments. \\
(2) Note that the collection of the labels on all the segments of the strands $\gamma_i$ is multiplicity-free. 
In general, it is a proper subset of $\{1,2,\dots, n\}$. \\
(3) Consider an alternating region which is ``cut'' by $\gamma$. Associate labels to the two halves of 
this region (under the cut by $\gamma$) according to steps 4 and 5. Let $\{i_1,\dots, i_r\}$ be the labels of the 
region clockwise from $\gamma$. Then the labels of the other half are 
$\{i_1+n_0,\dots, i_r+n_0\}$. 
\end{rmk}

Comparing the above construction with the definition of labels for orbifold diagrams, we get: 
\begin{lemma}
The set of labels for $\OO$ obtained through Algorithm~\ref{alg:labels} is a system of representatives 
for $I_\OO$. 
\end{lemma}


We illustrate the algorithm on the three running examples to show how we 
associate labels to orbifold diagrams. 

\begin{ex}\label{ex:labels-d-5} 
	In Figure \ref{Labels410}, we apply Algorithm~\ref{alg:labels} 
	to the orbifold diagram of Example~\ref{ex:ex2-ex1-orb} (2). Recall that $d=5$ and $n=10$. 
	The labels to consider in Step 5 are $3,5,10$. Only $10$ satisfies the condition of Step 5 and will 
	get added to all alternating regions. 
	\begin{figure}[H]
		\includegraphics[scale=.45]{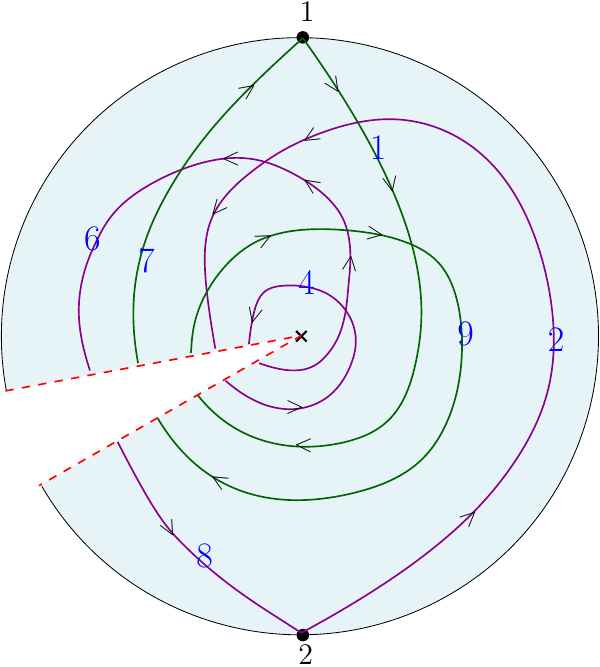}
		\hspace{.5cm}
		\includegraphics[scale=.45]{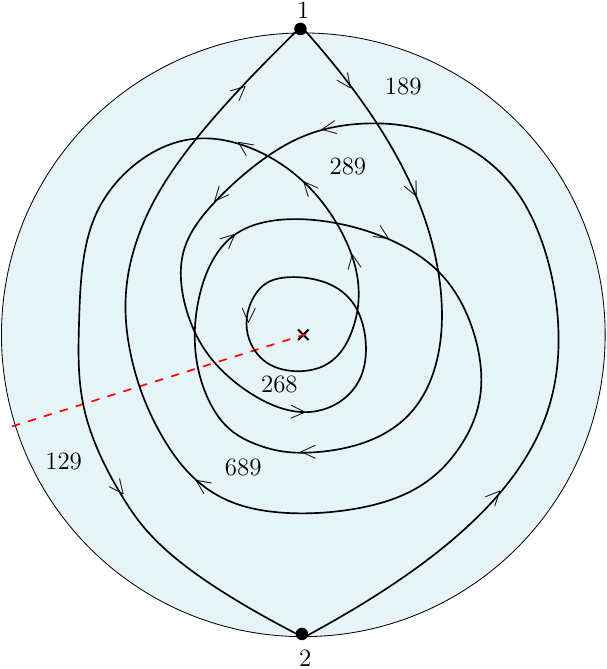}
		\hspace{.5cm}
		\includegraphics[scale=.45]{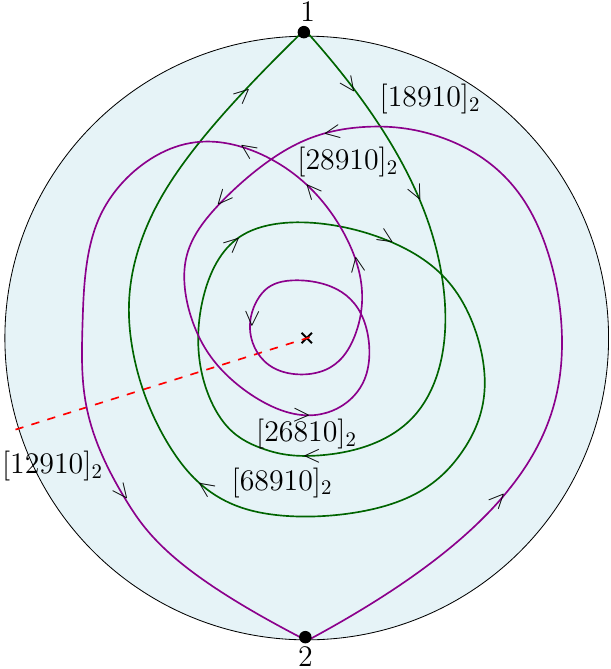}
		\caption{Steps 1 and 2 of the algorithm are illustrated on the left. Steps 4,5 in the middle and 
		the resulting set of labels is drawn on the orbifold diagram on the right.}
		\label{Labels410}
	\end{figure}	
\end{ex}

\begin{ex}\label{ex:labels-d-3}
	In Figure \ref{Labels3n0}, we apply the Algorithm~\ref{alg:labels} to Example \ref{ex:ex2-ex1-orb} (1) 
	with $d=3$, $n=9$. The labels to consider in Step 5 are $5,7,9$. Both $7$ and $9$ satisfy the condition and 
	will get added to all alternating regions. 
	\begin{figure}[H]
		\includegraphics[scale=.65]{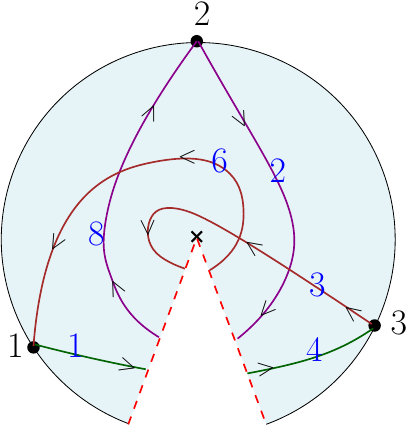}
		\hspace{0.6cm}
		\includegraphics[scale=.65]{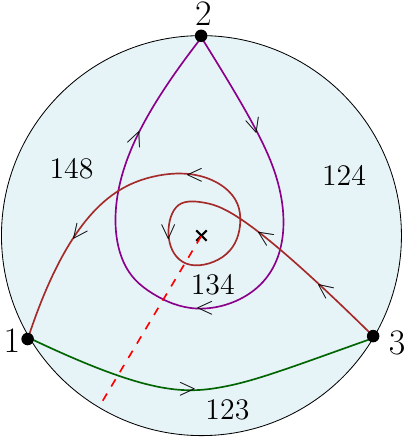}
		\hspace{0.6cm}
		\includegraphics[scale=.65]{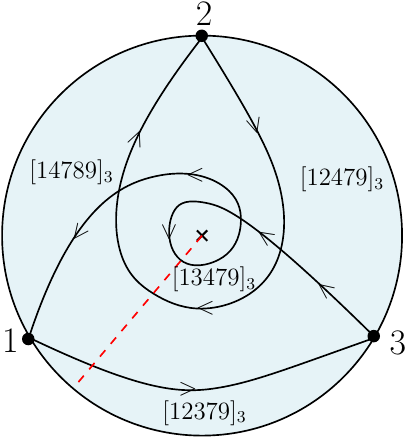}
		\caption{Steps 1 and 2 are depicted on the left, steps 3 and 4 
		in the center and the last step on the right.}
		\label{Labels3n0}
	\end{figure}
	
\end{ex}

\begin{ex}\label{Ex:CentralFixedPoint}
We consider the orbifold diagram $\OO$ of order 3 from Example~\ref{ex:middle-alternating} and want to 
determine its labels.  
In this example we have an alternating region around the orbifold point of 
order 3. We are going to use Definition~\ref{defin:I_O}. On the left hand in  
Figure~\ref{fig:NewLabelEx-1} we have the labels for the Postnikov diagram $\sym_3(\mathcal{O})$, 
the dashed lines indicate three fundamental regions for the action, namely the rotation by $\frac{2\pi}{3}$.  
For each alternating region of the 
diagram on the left in Figure~\ref{fig:NewLabelEx-1}, we assign an equivalence class of $3$-element 
subsets  of $\{1,2, \ldots, 9\}$ by considering a representative given by the label associated with the 
fundamental region containing the vertices  $\{2,3,4\}$ of $\sym_3(\mathcal{O})$ on the right of the same 
Figure. Note that all but the region containing the orbifold point have an equivalence class with three 
elements, one for each copy of the fundamental region. 
The region containing the orbifold point has an equivalence class with just one element because  
the corresponding region in $\sym_3(\mathcal{O})$ is fixed by the action.
\end{ex}

\begin{figure}[ht]
	\includegraphics[scale=.58]{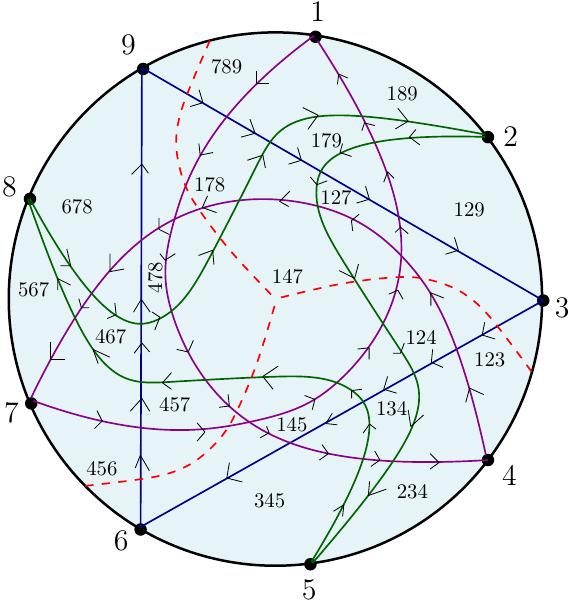}
	\hskip 1cm
	\includegraphics[scale=.6]{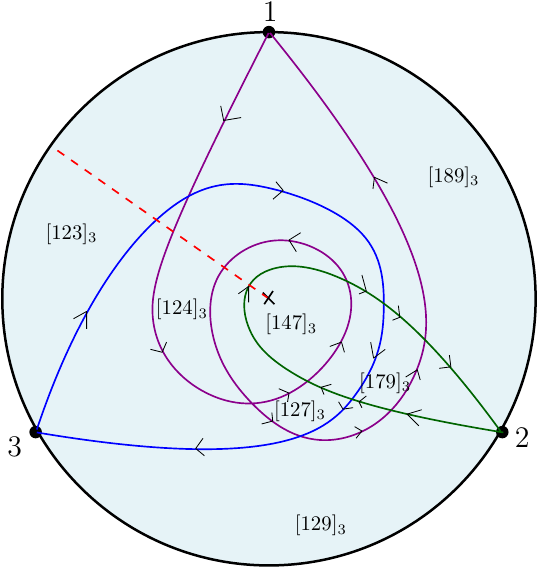}  
	\caption{The first figure shows the labels of $\sym_3(\mathcal O)$, the second figure 
	shows the equivalence classes of the labels for $\OO$  under the equivalence relation $\sim_{3}$}.
	\label{fig:NewLabelEx-1}
\end{figure}

\section{Quivers with potentials}

Now we shall define a quiver with potential (QP for short) associated to an orbifold diagram, in order to 
compare it with the one associated to its cover as in~\cite[Section 3]{BKM16}.

\begin{defin}
	Let $\mathcal P$ be a Postnikov diagram. We associate to it a quiver with potential 
$(Q_{\mathcal P}, W_{\mathcal P})$ as follows. The vertices of $Q_{\mathcal P}$ are given by 
the alternating regions of $\mathcal P$ (recall that we treat boundary regions as alternating). 
For any two alternating regions sharing a crossing, there is an arrow in $Q_{\mathcal P}$ going 
through that crossing following the orientation of the strands.
	Observe that $Q_{\mathcal P}$ is naturally a quiver with faces, with fundamental 
cycles corresponding to cyclical regions of $\mathcal P$. The potential $W_{\mathcal P}$ is defined 
as the sum of these cycles, with signs depending on their orientations. 
\end{defin}

\begin{ex}
In Figure~\ref{fig:dimer-relations}, we have the quiver $Q_{\mathcal P}$ for the Postnikov diagram 
(with labels) from Figure~\ref{fig:P-examples}. On the right, the quiver is drawn with straight arrows. 
The potential $W_{\mathcal P}$ is 
\[
\sum_{i=1}^7\on{sgn}(c_i)c_i = c_1-c_2+c_3-c_4+c_5-c_6+c_7,
\]
for $c_i$ the fundamental cycle of the face indicated by $c_i$, taken with $+$ if and only if $c_i$ is counterclockwise. 
\begin{figure}
\includegraphics[scale=.65]{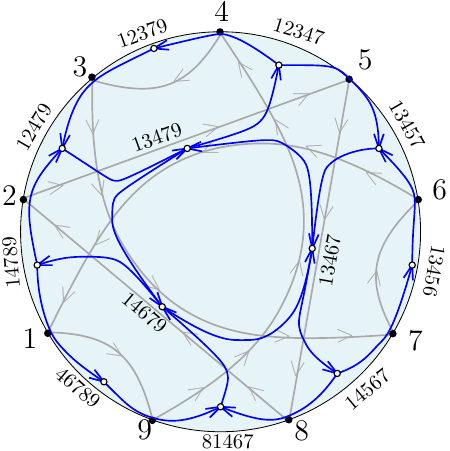}
\hskip .6cm 
\includegraphics[scale=.65]{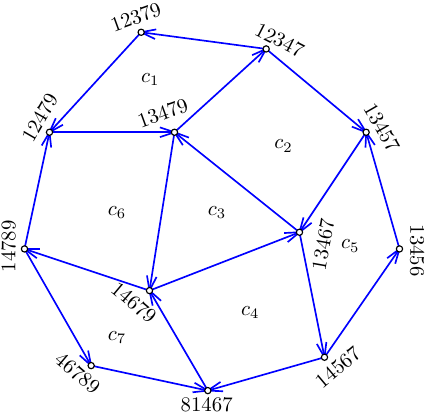}
\caption{The quiver $Q_{\mathcal P}$ for $\mathcal P=\sym_3(\OO)$ from Example~\ref{Exam2}, 
straightened on the right.} 
\label{fig:dimer-relations}
\end{figure}
\end{ex}

The quiver $Q_P$ is called a dimer model with boundary in~\cite{BKM16}. We recall the definition of 
the (frozen) Jacobian algebra associated to a quiver with potential: 
\begin{defin}\label{def:frozen-jacobi}
Let $(Q_{\mathcal P},W_{\mathcal P})$ be the quiver with potential associated to the Postnikov diagram 
$\mathcal P$. The {\em frozen Jacobian algebra associated to the QP} $(Q_{\mathcal P},W_{\mathcal P})$ 
is the completed 
path algebra of $Q_{\mathcal P}$ modulo the closure of the relations given by the cyclic derivatives of the 
potential with respect to {\em internal arrows} (arrows incident with two faces): 
Let $\alpha$ be internal and let $c_1$ and $c_2$ be the two fundamental cycles containing $\alpha$, 
with $\on{sgn}c_1=+$, $\on{sgn}c_2=-$. We write $c_i=\alpha c_i'$ for $i=1,2$. 
Then $\partial_{\alpha}W_{\mathcal P}=c_1'-c_2'$. 
In other words, for any internal arrow $\alpha$, the two paths completing $\alpha$ to a fundamental cycle 
agree. For example, in Figure~\ref{fig:dimer-relations}, the arrow $14679\to 12467$ induces as a 
relation that the path of length 2 from $12467$ to $14679$ is equal to the path of length 3 from 
$12467$ to $14679$. 
\end{defin}

\begin{defin}	\label{def:frozen-jacobi-2}
Now assume that $\mathcal P$ is reduced, i.e.~no reduction moves such as 
in Figure~\ref{fig:PullingStrand} are possible. Then we define the {\em algebra $A(\mathcal P)$ 
of $\mathcal P$} 
as the (completed) frozen Jacobian algebra of $(Q_{\mathcal P}, W_{\mathcal P})$, where the frozen 
vertices correspond to the boundary regions. If $e$ is the idempotent corresponding to the 
vertices of the boundary regions, 
we also define the {\em boundary algebra} $B(\mathcal P)$ {\em of $\mathcal P$} 
to be the idempotent subalgebra $eA(\mathcal P)e$.
\end{defin}

We will give an analogous definition of quiver with potential for orbifold diagrams, in such a way that the 
frozen Jacobian algebras are related to each other by a skew group algebra construction. This requires 
some work.

\begin{defin}\label{Def:Quiver_O}
Let $\OO$ be an orbifold diagram of order $d$. We associate to it a quiver $Q_\OO$ as follows. 
The vertices of $Q_\OO$ are given by the alternating regions of $\OO$ (including the regions on 
the boundary). 
If the orbifold point $\Omega$ is contained in an alternating region, we associate to that region $d$ vertices 
$v_1,\dots, v_d$ of $Q_\OO$. We imagine the $v_i$ as lying on a line orthogonal to the disk above the 
orbifold point. 
For any two vertices which are separated by a crossing of oriented strands, 
there is an arrow in $Q_\OO$ going through that crossing following the orientation of the strands. 
In the case of vertices $v_1,\dots, v_d$ (if present), we draw arrows 
between each of them and all the neighbouring regions but no arrows between these vertices. 
\end{defin}

In case the region containing $\Omega$ is alternating, with an even number $r\ge 2$ of 
arrows incident with it, then each of the vertices $v_i$ has $r$ fundamental cycles incident with it. 

The quiver $Q_\OO$ is also naturally a quiver with faces: 
Its fundamental cycles correspond to 
cyclical regions in $\OO$ which do not involve the orbifold point, together with with $d$ copies of the 
$r$ cycles corresponding to cyclical regions adjacent to the central region containing $\Omega$, 
if this region is alternating. 
Seen as a CW-complex, this quiver with faces consists in this case 
of an annulus where the boundaries are given by non-oriented cycles, together 
with $d$ disks. These disks are all isomorphic as quivers with faces and their boundary cycle 
(which is in general not oriented) is 
identical to the inner boundary cycle of the annulus. These boundary cycles are 
identified with the inner boundary of the annulus, i.e. the $d$ disks are all glued along one of the 
boundary components of this annulus, see~\cite[Proposition~7.7]{GP19}.
If the region containing $\Omega$ is cyclical then
$Q_\OO$, as quiver with faces, is a tiling of the disk.

In what follows, if $\Omega$ belongs to an alternating region, 
we write $c_i^{(1)},\ldots, c_i^{(r)}$ for the $r$ fundamental cycles in 
$Q_\OO$ through $v_i$, for $i=1,\dots, d$. The labeling is done in a way such that the $c_i^{(r)}$ is the unique 
cyclical region adjacent to the central alternating region and intersecting the curve $\gamma$. 

Note that all the fundamental cycles come with an orientation and hence with a sign: We set 
$\on{sgn}(c)$ to be $1$ if $c$ is a counterclockwise fundamental cycle and $-1$ if $c$ is clockwise. 
Then we can define a potential for the quiver $Q_\OO$. 

\begin{defin}\label{Def:Potential_O}
	Let $\OO$ be an orbifold diagram of order $d$ and let $Q_\OO$ be its quiver. 
Let $\mathcal{C}$ be the set of the fundamental cycles of $Q_\OO$. 
We define a potential $W_\OO$ on $Q_\OO$ as follows. 
	\begin{itemize}
		\item Assume that the orbifold point $\Omega$ lies in a cyclical region and 
		let $c$ be the corresponding fundamental cycle. 
		We set 
		\begin{align*}
		W_\OO = \frac{1}{d}\on{sgn}(c)c^d+ \sum_{c'\in\mathcal{C}\setminus \{c\}} \on{sgn}(c')c'.
		\end{align*}
		\item Assume that $\Omega$ lies in an alternating region and let 
$\mathcal{C'}$ the set of all the $c_i^{(j)}$. 
Fix a primitive $d$-th root of unity $\zeta$. 
We set 
		\begin{align*}
		W_\OO = \sum_{c\in\mathcal{C}\setminus \mathcal{C'}} \on{sgn}(c)c+\sum_{j\ne r}\sum_{i=1}^d \on{sgn}(c_i^{(j)})c_i^{(j)}+\sum_{i=1}^d\zeta^i\on{sgn}(c_i^{(r)})c_i^{(r)}.
		\end{align*}
	\end{itemize}
\end{defin}

\begin{rmk}
	The above definition of $W_\OO$ depends on the choice of $\zeta$. However, Theorem~\ref{thm:main} shows that the frozen Jacobian algebras corresponding to different choices are isomorphic.
\end{rmk}

\begin{rmk}
	The potential $W_\OO$ of Definition~\ref{Def:Potential_O} is equal, for suitable choices (see the proof of  Proposition~\ref{prop:sgas}), to the potential $W_G$ of \cite[Notation~3.18]{GP19} divided by $d$. It is also equal to the potential $W_G$ of \cite[Definition~5.3]{GPP19}.
\end{rmk}

\begin{defin}
	For an orbifold diagram $\mathcal O$, define the {\em algebra $A(\mathcal O)$ of $\OO$} 
as the frozen Jacobian algebra of $(Q_\OO, W_\OO)$, with frozen vertices the boundary vertices. 
If $e$ is the idempotent corresponding to the boundary vertices, we define the 
{\em  boundary algebra} $B(\mathcal O)$ {\em of $\OO$} to be 
the algebra $B(\mathcal O)= eA(\mathcal O)e$.
\end{defin}	

Note that whenever for $\OO$, the orbifold point $\Omega$ is contained in a cyclic region, 
we have a new type of 
terms in the relations for $A(\OO)$: In this case, the cycle (or loop) appears as term 
$c^d$ in the potential. Taking derivatives with respect to arrows of this cycle (with respect to the 
arrow of the loop) gives a $d$-fold term in the relations for $A(\mathcal O)$ (with this $d$ cancelling out the $\frac{1}{d}$ coefficient). For the 
quiver with potential in Figure~\ref{fig:quiver-O}, 
taking the derivative with respect to the loop arrow $c$ 
gives $c_1'=c^2$, where $c_1'$ is the path $[14679]_3\to [14789]_3\to [12479]_3$.

\begin{ex}
	We illustrate Definitions \ref{Def:Quiver_O} and \ref{Def:Potential_O} on the orbifold diagram 
$\mathcal{O}_1$ of order 3 from Example~\ref{ex:ex2-ex1-orb} (1) with labels in 
Example~\ref{ex:labels-d-3} and on the orbifold diagram $\mathcal{O}_2$ from 
Example~\ref{ex:ex2-ex1-orb} (2) with labels in Example~\ref{ex:labels-d-5}. 
The quivers $Q_{\mathcal{O}_1}$ and $Q_{\mathcal{O}_2}$ are depicted 
in Figure~\ref{fig:quiver-O} and Figure~\ref{fig:quiver-1}, respectively. 
%
 \begin{align*}
 W_{\mathcal{O}_1}&=-c_1+c_2+\frac{1}{3}c^3,\\
 W_{\mathcal{O}_2}&=-c_1+c_2+c_3-c_4+\frac{1}{5}c^5. 
 \end{align*}
	\begin{figure}[H]
		\includegraphics[scale=.63]{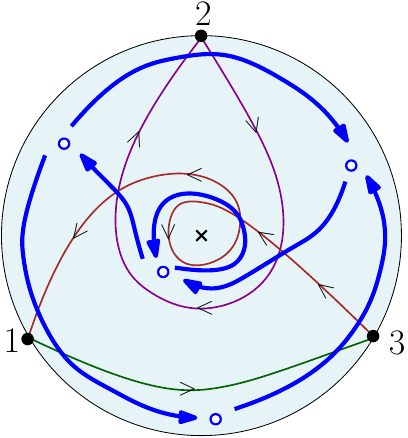}
		\hspace{.5cm}
		\includegraphics[scale=.63]{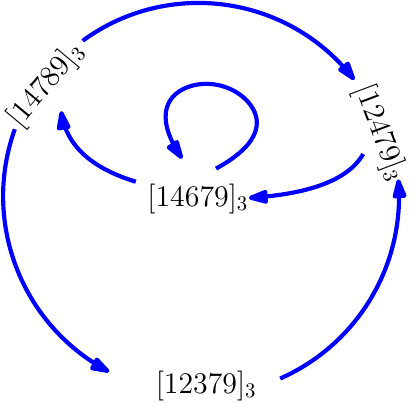}
		\hspace{.5cm}
		\includegraphics[scale=.63]{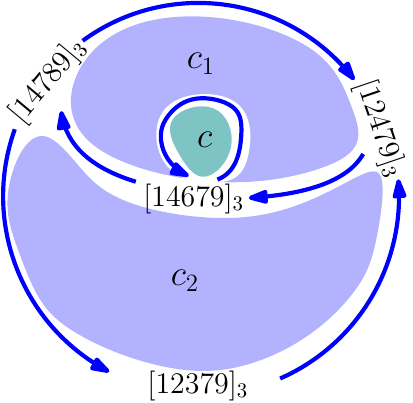}
		\caption{The quiver $Q_{\mathcal{O}_1}$ associated to the orbifold diagram of 
		Examples~\ref{ex:ex2-ex1-orb}(1).}
		\label{fig:quiver-O}
	\end{figure}
\end{ex}

\begin{figure}[H]
	\includegraphics[scale=.42]{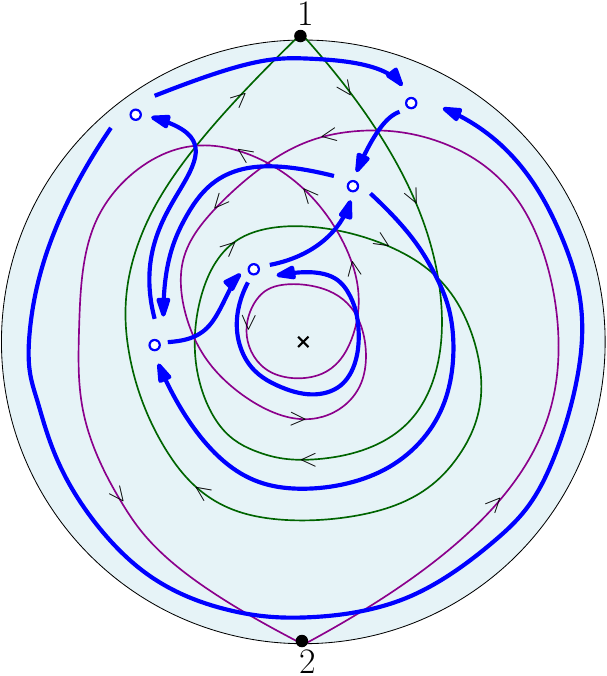}
	\hspace{.5cm}
	\includegraphics[scale=.42]{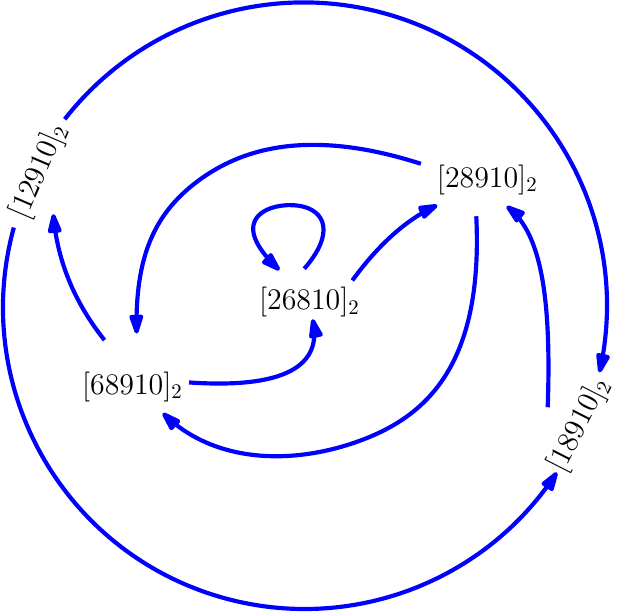}
	\hspace{.5cm}
	\includegraphics[scale=.42]{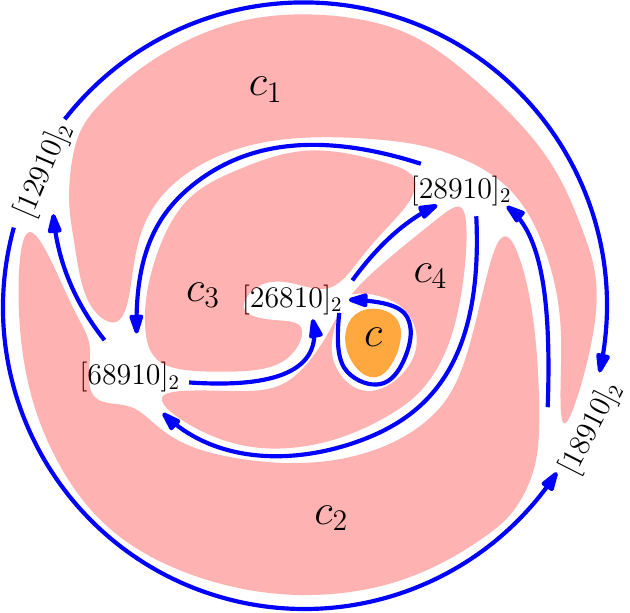}	
	\caption{The quiver $Q_{\mathcal{O}_2}$ associated to the orbifold diagram of 
	Examples~\ref{ex:ex2-ex1-orb}(2).}
	\label{fig:quiver-1}
\end{figure}


\begin{figure}[H]
	\includegraphics[scale=.6]{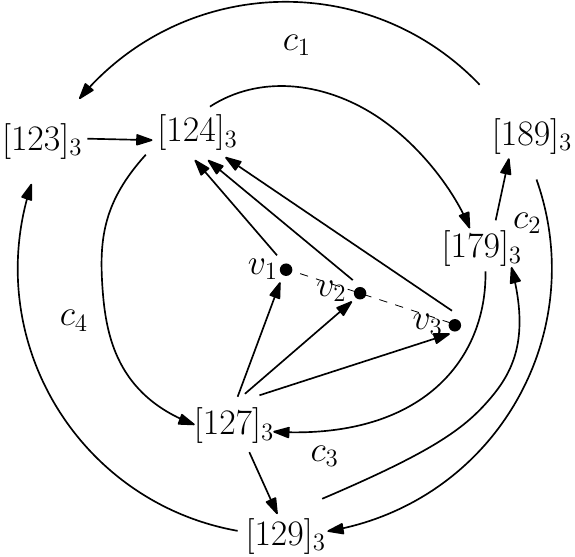}  
	\caption{The quiver $Q_{\mathcal{O}}$ associated to an orbifold diagram $\mathcal{O}$ of order 3 with the orbifold point inside an alternating region and the fundamental cycles $c_i$ far away from the orbifold point.}
	\label{fig:QuiverFixedPoint}
\end{figure}

\begin{ex}\label{Ex:QPCentralFixedPoint}
Let us consider the quiver with potential for the orbifold diagram $\mathcal{O}$ given in 
Example \ref{Ex:CentralFixedPoint}. Recall that this example is particularly interesting because it has an 
alternating region containing the orbifold point. We are going to follow Definition \ref{Def:Quiver_O}. 
The quiver $Q_{\mathcal{O}}$ is depicted in Figure \ref{fig:QuiverFixedPoint}. Note that on the right of 
Figure \ref{fig:NewLabelEx-1} the alternating region containing the orbifold point got the label 
$[147]_3$. Following Definition \ref{Def:Quiver_O}, the label $[147]_3$ yields three vertices 
called $v_1$, $v_2$ and $v_3$ on $Q_{\mathcal{O}}$, see Figure \ref{fig:QuiverFixedPoint}. 
Since the alternating region containing $\Omega$ is given by 2 arrows, we have 
two cyclic regions around this region /with label $[147]_3$). One 
is a triangle given by $\{[124]_3, [127]_3, [147]_3\}$, 
the other one is a quadrilateral given by $\{[127]_3, [147]_3, [124]_3, [179]_3 \}$. 
We write $c_i^1$ ($i=1,2,3$) to denote the three fundamental cycles arising from the triangle at 
$[147]_3$ and $c_i^2$ for the three fundamental cycles arising from the quadrilateral. 
These six faces are illustrated in Figure~\ref{fig:FacesFixedPoint} by different shadings. 
The labeling of the other faces is given in Figure~\ref{fig:QuiverFixedPoint}. 
We thus have the potential 
\[
W_{\mathcal{O}}= +c_1-c_2+c_3-c_4+c_1^{(1)}+c_2^{(1)}+c_3^{(1)}-\zeta c_1^{(2)}-\zeta^2c_2^{(2)}-c_3^{(2)},
\]
where $\zeta$ is a primitive third root of the unity. 
\end{ex}

\begin{figure}[ht]
	\includegraphics[scale=.55]{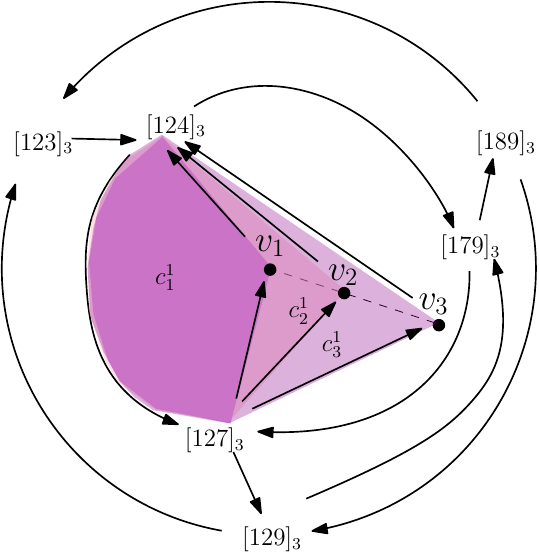}\hspace{0.8cm}
	\includegraphics[scale=.55]{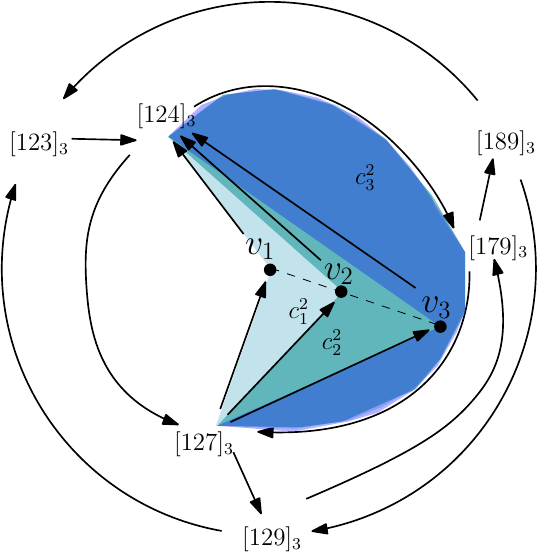}  
	\caption{The fundamental cycles around the orbifold point on $Q_{\mathcal{O}}$.}
	\label{fig:FacesFixedPoint}
\end{figure}


\section{Skew group algebras from orbifold diagrams}

In this section we recall the notion of a skew group algebra and prove properties which we will need later. 
We want to relate the algebras $A(\OO)$ and $B(\OO)$ of an orbifold diagram $\OO$ 
with the algebras $A({\sym}_d(\OO))$ and $B({\sym}_d(\OO))$ of the associated 
symmetric Postnikov diagram. 
In particular, we want to describe endomorphism algebras (see Lemma~\ref{lm:morita-equ} 
and Lemma~\ref{lem:isom-1}). This will be a key ingredient for 
Section~\ref{sec:categories} where we study the associated module categories. 

\begin{defin}\label{def:skew-group}
	Let $S$ be an algebra with an action of a finite group $G$ by automorphisms. 
The \emph{skew group algebra} $S*G$ is $S\otimes_\C \C G$ as a vector space, with multiplication 
linearly induced by  $$(s\otimes g)(t\otimes h)=s g(t)\otimes gh$$
(where $s,t\in S$, $g,h\in G$). 
\end{defin}

The group $G$ acts on the category $\on{mod}(S)$ (left modules) by twists, that is 
$g(L) =\tw{g}{L}$ which is $L$ as a set, with $S$-action 
given by $s\cdot_{\tw{g}{L}} l = g(s)l$, noting that for all $g,h\in G$, we have $\tw{g}(\tw{h}L)=\tensor[^{hg}]{L}{}$. 
This gives an autofunctor of $\on{mod}(S)$ by letting $G$ act trivially on morphisms. To simplify the 
notation, we will write morphisms as $f$ instead of $\tw{g}{f}$. 

There is an induction functor $F$ from $\on{mod}S$ to $\on{mod}(S*G)$ sending $M$ 
to $FM:=(S*G)\otimes_S M$. 

The \emph{category of $G$-equivariant $S$-modules} $\on{mod}(S)^G$ is defined to have 
as objects the pairs 
$\left(L, \left(\varphi_g\right)_{g\in G}\right)$, where $L$ is an object of $\on{mod}(S)$ and 
where the $\varphi_g: L\to \tw{g}{L}$ are isomorphisms satisfying the following:
\begin{enumerate}
	\item  $\varphi_{hg}$ $= \varphi_h\circ \varphi_g$,
	\item $\varphi_1 = \on{id}_L$.
\end{enumerate}
Morphisms in $\on{mod}(S)^G$ are morphisms of $S$-modules which intertwine the $\varphi_g$. Precisely, if $\left(L, \left(\varphi_g\right)_{g\in G}\right)$ and $\left(N, \left(\psi_g\right)_{g\in G}\right)$ are in $\on{mod}(S)^G$, then a map $f:L\to N$ of $S$-modules is a morphism in $\on{mod}(S)^G$ if there is a commutative diagram
\[\xymatrix{
  L \ar^f[r] \ar[d]_{\varphi_g}& N\ar[d]^{\psi_g} \\
  \tw{g}{L} \ar^f[r] & \tw{g}{N}
}\]
for every $g\in G$.

In what follows, we will use the fact that the category of finitely generated $S*G$-modules is equivalent to the 
category $\on{mod}(S)^G$, as an $S*G$-module can be viewed as as $S$-module with a compatible $G$-action.

%

From now on, let us assume that $G$ is finite of order $d$. 
We will use the skew-group construction for $G$ 
and the algebras associated to orbifold diagrams (of order $d$). 
On one hand, we will have modules which are invariant under the 
group action and on the other hand modules, which are permuted by the elements of $G$. 
In Section~\ref{sec:categories},  
we will use the effect of the $F$ on modules which are built up from these two types, 
i.e. we will need to study the functor $F$ on a direct sum $M_0\oplus M$ where $M_0$ is invariant 
under the group action and where $M$ is a sum of isomorphic summands for each group element 
(see Lemma~\ref{lem:iso-representatives}).

Let $M_0\in \on{mod}(S)$ be a module such that $\tw{g}{M}_0=M_0$ for all $g\in G$.

Then 
$$
FM_0= \left(\bigoplus_{g\in G}M_0, (\varphi_h)_{h\in G}\right) = 
\left(\bigoplus_{\sigma\in G^*}M_0, (\lambda_h)_{h\in G}\right)=\bigoplus_{\sigma\in G^*}(M_0, (\sigma(h)\on{id})_{h\in G} ),$$ 
where $(\lambda_h)_{\sigma, \sigma} =\sigma(h)\on{id}_{M_0}$ 
and the other entries of $\lambda_h$ are zero. 
This is because the subset 
\[
\left\{x\in \bigoplus_{g\in G}M_0\ \left.\right\vert\ h(x) = \sigma(h)x\right\}\subseteq \bigoplus_{g\in G}M_0
\]
is isomorphic to $M_0$ as an $S$-module (intuitively, we are decomposing the regular representation into the irreducible characters of $G$).

Let $M_1\in \on{mod}(S)$ be any module. Define $M_g = \tw{g}{M}_1$ for every $g\in G$, and 
$M= \bigoplus_{g\in G}M_g$. Then define $\varphi_g:M\to \tw{g}{M}$ to be the canonical isomorphism 
that permutes the summands, i.e.~$\varphi_g(M_h)= \tw{g}{M}_{g^{-1}h}$. This makes 
$\left(M, \left(\varphi_g\right)_{g\in G}\right)$ an object in $\on{mod}(S)^G$. 
As a module over $S*G$, it is $M$ 
with $b\otimes g$ 
acting by $\varphi_g\circ b\cdot\circ \varphi_g^{-1}$. 
We will just write $M$ for this object of 
$\on{mod}(S*G)$ as well as for the $S$-module $M$.

\begin{rmk}\label{rem:module-SG}
Observe that we really mean $M$ as an $(S*G)$-module, and not $FM$. Indeed, we have 
$FM\cong M^{\oplus d}$, and in fact $M= FM_1$ naturally. 
As vector spaces, $M=\mathbb{C}G\otimes_{\mathbb{C}} M_1$. 
\end{rmk}


\begin{lemma}\label{lem:isom-1}
With the notation as above, there is an algebra isomorphism 
\[
\on{End}_S(M_0\oplus M)*G \cong \on{End}_{S*G}(FM_0\oplus FM)
\]
which maps any $f\otimes g$ to 
\[
\begin{array}{lcl}
F(M_0\oplus M) & \to & F(M_0\oplus M) \\
(s\otimes h)\otimes m & \mapsto & (s\otimes hg^{-1})\otimes f(\varphi_g(m)) 
\end{array}
\]
\end{lemma}

\begin{proof}
We will show that this is a composition of two vector space isomorphisms $\theta_1$ and $\theta_2$ which 
we now define. First consider 
\[
\theta_1:\on{End}_S(M_0\oplus M)*G \to \on{Hom}_S(M_0\oplus M,F(M_0\oplus M)).
\]
%
The map $\theta_1$ can be obtained as follows: 
we send $f\otimes g\in \on{End}_S(M_0\oplus M)\otimes \C G$ to the homomorphism 
\[
\begin{array}{ccl}
M_0\oplus M & \to & F(M_0\oplus M) \\
m & \mapsto & (1\otimes g^{-1})\otimes f(\varphi_g(m)). 
\end{array}
\]
in $\on{Hom}_S(M_0\oplus M,F(M_0\oplus M))$. 
Then consider 
\[
\theta_2:\on{Hom}_S(M_0\oplus M,F(M_0\oplus M)) \to \on{End}_{S*G}(F(M_0\oplus M))
\]
This is the adjunction isomorphism of vector spaces which maps any 
$f\in \on{Hom}_S(M_0\oplus M, F(M_0\oplus M))$ to 
\[
\begin{array}{lcl}
F(M_0\oplus M)  & \to & F(M_0\oplus M) \\
(s\otimes g)\otimes m & \mapsto & (s\otimes g)\cdot f(m) 
\end{array}
\]

\noindent
The composition $\theta_2\circ\theta_1$ sends $f\otimes g$ to 
$(s\otimes h)\otimes m \ \mapsto \ (s\otimes hg^{-1})\otimes f(\varphi_g(m))$. 
%
%
Furthermore, the fact that the composition is multiplicative is a direct check. 
So $\theta_2\circ\theta_1$ is an 
algebra homomorphism. 
\end{proof}

Now we introduce two homomorphisms of $S*G$-modules.

\begin{defin} 
(1) 
Let $i:FM_0\oplus M\to FM_0\oplus FM$ be the following map: 
\[
(m_0,m) \in FM_0\oplus M \longmapsto \left(m_0,\frac{1}{|G|}\sum_{g\in G}(1\otimes g)\otimes\varphi_g^{-1}(m)\right)
\]

(2) 
Let $p:FM_0\oplus FM\to FM_0\oplus M$ be the following map 
\[
\left( m_0,\sum_{g\in G}(s_g\otimes g)\otimes m_g\right) \longmapsto 
\left( m_0,\sum_{g\in G} s_g \cdot \varphi_g(m_g)\right)
\]

\end{defin}

The proof of the following follows immediately from the definition. 

\begin{lemma}\label{lem:maps-p-i}
The maps $i$ and $p$ are homomorphisms of $S*G$-modules. 
Furthermore, their composition $p\circ i$ is the identity homomorphism. 
\end{lemma}

Next we set $e$ to be 
the preimage under $\theta_2\circ\theta_1$ of the element $i\circ p$ of 
$\on{End}_{S*G}(FM_0\oplus FM)$. This is an idempotent in $\on{End}_S(M_0\oplus M)*G$ 
because $i\circ p$ is an idempotent of $\on{End}_{S*G}(FM_0\oplus FM)$. 

\begin{lemma}\label{lm:old-lemma-5-3}
There is an algebra isomorphism 
$\on{End}_{S*G}(FM_0\oplus M)\cong e\cdot (\on{End}_S(M_0\oplus M)*G)\cdot e$. 
\end{lemma}

\begin{proof}
This follows using the maps 
\[
\begin{array}{rccc}
\beta_1: & \on{End}_{S*G}(FM_0\oplus FM) & \longrightarrow & \on{End}_{S*G}(FM_0\oplus M) \\ 
 & f & \longmapsto & p \circ f \circ i \\
 & \\ 
\beta_2: & \on{End}_{S*G}(FM_0\oplus M) & \longrightarrow & \on{End}_{S*G}(FM_0\oplus FM) \\ 
 & f' & \longmapsto & i \circ f'\circ p 
\end{array}
\]
with $(\beta_1\circ\beta_2)(f')=f'$ and $(\beta_2\circ\beta_1)(f)=i\circ p\circ f\circ i\circ p$ and since 
$\on{End}_{S}(M_0\oplus M)*G\cong \on{End}_{S*G}(FM_0\oplus FM)$ by 
Lemma~\ref{lem:isom-1}. 
\end{proof}

Then since $FM$ is isomorphic to $M^{|G|}$ as a $S*G$-module, the algebras 
$\on{End}_{S*G}(FM_0\oplus M)$ and $\on{End}_S(M_0\oplus M)*G$ are Morita equivalent: 

\begin{lemma}\label{lm:morita-equ}
With the notation as above, we have 
$\on{End}_{S*G}(FM_0\oplus M) \sim \on{End}_S(M_0\oplus M)*G$. 
\end{lemma}

\section{Characterising the algebras arising from orbifold diagrams} 
\label{sec:categories}

In this section we combine the results of the previous sections to characterise the algebras 
$A(\OO)$ and $B(\OO)$.
From now on, we assume that $\OO$ is a reduced orbifold diagram on a disk with $n_0$ marked points 
and that its cover ${\sym}_d(\OO)$ 
is a (reduced) Postnikov diagram on a disk with $n=n_0d$ marked points.
By Proposition~\ref{prop:orbifold-is-quotient}, this is the case as soon as $d>2$. 
There are examples of orbifold diagrams of order 2 where ${\sym}_2(\OO)$ is also a Postnikov diagram and the results in this 
section hold in this case. 

Let $\mathcal P$ be a $d$-symmetric Postnikov diagram. Let $G$ be the cyclic group generated 
by clockwise rotation by $\frac{2\pi}{d}$. This group acts on both $A(\mathcal{P})$ and 
$B(\mathcal{P})$ by automorphisms in a natural way, and exactly this group and its action that we 
fix when we take skew group algebras. 
Before restricting to the Grassmannian setting, we present a general result. 

We will use the construction of the quiver with potential of a skew group algebra given 
in \cite{GP19}. For convenience of the reader, we summarize here some points about 
this construction.
If $(Q, W)$ is a QP and $G$ is a finite group acting on $Q$ fixing $W$, it is known 
by the work of Le Meur, \cite{leM20}, that the skew group algebra of the Jacobian algebra of $(Q, W)$ is 
Morita equivalent to the Jacobian algebra of a new QP $(Q_G,W_G)$. 
(We will use $\sim$ below to denote Morita equivalence). 
Under certain assumptions which are satisfied in the case of a symmetric Postnikov diagram 
with $G$ acting by rotations, one can describe $(Q_G, W_G)$ explicitly. The quiver $Q_G$ 
was constructed in \cite{RR85}, and the potential $W_G$ is defined in \cite[Notation~3.18]{GP19} after making certain choices, notably: a set of representatives of vertices of $Q$, and a 
suitable set of representatives of cycles appearing in $W$. The potential $W_G$ depends 
on these choices (and on the choice of a primitive root of unity), but the resulting Jacobian algebras 
are isomorphic.

\begin{prop}\label{prop:sgas}
	Let $\OO$ be an orbifold diagram of order $d$. 
Then we have the Morita equivalences 
$$
A(\OO)\sim A({\sym}_d(\OO))*G
$$ 
and 
$$
B(\OO)\sim B({\sym}_d(\OO))*G.
$$
\end{prop}

\begin{proof}
	We will use Theorem~3.20 of~\cite{GP19}, with $\Lambda$ being $A({\sym}_d(\OO))$. 
We remark this theorem still works if we replace the usual definition of the Jacobian ideal 
by any ideal generated by cyclic derivatives with respect to 
arrows (see Definition~\ref{def:frozen-jacobi})  
provided that these arrows are closed under the $G$-action. 
In particular, the statement immediately extends to the 
case of frozen arrows, which we are considering in our definition of $A(\OO)$. 
In our case, the frozen arrows are the boundary arrows, which indeed form a set closed 
under the $G$-action.

Moreover, the $G$-orbits of the boundary arrows for 
$A({\sym}_d(\OO))$ correspond exactly (in the sense of~\cite[Notation 3.13]{GP19}) 
to the boundary arrows of $A(\OO)$. It follows that even in our case, it is enough to show that 
the QP $(Q_G, W_G)$ of~\cite{GP19} is equal to $(Q_\OO, W_\OO)$, if we make appropriate choices. 
The fact that $Q_G= Q_\OO$ is clear, as the 
two constructions both agree with the general construction presented 
in~\cite[Section 2]{RR85}. 
This is also illustrated in Examples 8.1 and 8.3 of~\cite{GP19}. 
	
	If the central region is cyclical, we are in the special case where $G$ acts freely on the whole 
quiver $Q$ of $\Lambda=A({\sym}_d(\OO))$  (and hence on $\Lambda$), which means that 
$\Lambda*G$ is Morita equivalent to the quotient $\Lambda/G$. In particular, the potential $W_\OO$ 
we have defined in this case makes $A(\OO)$ isomorphic to this quotient and we are done.
	
	It remains to check that the potential $W_\OO$ equals the potential $W_G$  
of~\cite[Notation~3.18]{GP19} (for appropriate choices) 
in the case where the central region is alternating. 
Following~\cite[\S3.2]{GP19}, we choose a set $\mathcal E$ of representatives of vertices of $Q$. 
In order to get the simple formulas we gave for $W_\OO$, we should be careful in how we choose the 
set $\mathcal E$. Let us pick a simple curve joining $\Omega$ to the boundary in $\OO$, draw its 
$d$ preimages under the quotient in ${\sym}_d(\OO)$, and consider one of the regions bounded by 
two consecutive copies, see for example Figure~\ref{fig:NewLabelEx-1}. 
We pick $\mathcal E$ to consist of exactly the vertices 
in this region. If the curve cuts an alternating region in two, we pick the part which is clockwise from the two 
copies of the simple curve and inside this region. 

We will now introduce some notation borrowed from~\cite{GP19}. Consider a cycle $c$ appearing in $W$, say with a scalar $a(c)$. Then there are two possbile cases:
\begin{itemize}
	\item either $c$ does not go through the central region (and all its vertices have trivial stabiliser), 
	\item or $c$ goes through the central region (and all its vertices except the one corresponding to the central region have trivial stabiliser).
\end{itemize}
Following~\cite[Notation~3.6]{GP19}, we say that $c$ is \emph{of type (i)} in the former case and \emph{of type (ii)} in the latter (we remark that \cite{GP19} treats additional cases which do not appear here).
Our construction will associate a summand in $W_G$ to every $G$-orbit of cycles appearing in $W$ (recall that $W$ is indeed $G$-invariant).
 By possibly applying the $G$-action, we can assume that $c$ is equal to 
$$
\xymatrix{
	I_0 = g^{t_1+\dots+t_{l}}(I_l) \ar[rr]^-{g^{t_1+\dots+t_{l-1}}(\alpha_l)} & &
	g^{t_1+\dots+t_{l-1}}(I_{l-1})\ar[r] & \cdots \ar[r]^-{g^{t_1}(\alpha_2)} & g^{t_1}(I_1)
	\ar[r]^-{\alpha_1} &I_0,
}
$$
where $\alpha_1,\dots,\alpha_l$ are arrows of $Q$, $t_1, \dots, t_l$ are integers, and $I_0, \cdots, I_l$ are vertices in $\mathcal E$. If $c$ is of type (i), this choice is not unique. If $c$ is of type (ii), then we can also assume that $I_1$ is the label of the central region (and then the choice is in fact unique).

Note that each $t_i$ is equal to $-1, 1$ or $0$ according to whether the arrow $\alpha_i$ crosses the cut $\gamma$ of step 1 of Algorithm~\ref{alg:labels} clockwise, counterclockwise or not at all, respectively.

The potential $W_G$ of~\cite[Notation~3.18]{GP19} is defined to be the sum of the contributions of all ($G$-orbits of) cycles appearing in $W$, as follows.
\begin{itemize}
	\item If $c$ is of type (i), then (each $\alpha_i\otimes g^{t_i}$ is an arrow of $Q_G$ and) its contribution to $W_G$ is $$
	a(c) (\alpha_1\otimes g^{t_1}) \cdots  (\alpha_l\otimes g^{t_l}).$$
	\item If $c$ is of type (ii), then:
	\begin{itemize}
		\item by our choice, we have $t_1=0$, 
		\item each $\alpha_i\otimes g^{t_i}$ is an arrow of $Q_G$ for $i\neq 1, 2$,
		\item for $\mu=0, \dots, d-1$, both $\alpha_1\otimes e_\mu$ and $g^{-t_2}(\alpha_2)\otimes e_\mu$ are arrows of $Q_G$, where $e_\mu$ is the element $$
		e_\mu = \frac{1}{d} \sum_{i=0}^{d-1}\zeta^{i\mu}g^i,$$
		\item the contribution of $c$ to $W_G$ is defined to be
		$$
		\sum_{\mu=0}^{d-1}a(c)\zeta^{-t_2\mu}(\alpha_1\otimes e_\mu)(g^{-t_2}(\alpha_2)\otimes e_\mu)(\alpha_3\otimes g^{t_3})\dots(\alpha_l\otimes g^{t_l}).$$
	\end{itemize}
\end{itemize}

The contributions of the cycles of type (i) to both $W_G$ 
and $W_\OO$ (where they correspond to $\mathcal C\setminus \mathcal C'$) are easily seen to agree, 
so it remains to check what happens with the cycles going through the middle (those giving rise to 
the cycles in $\mathcal C'$).
	
	The region between the two curves we chose on ${\sym}_d(\OO)$ contains $\frac{r}{2}$ outgoing 
and $\frac{r}{2}$ incoming arrows to the central vertex, so that $r-1$ cycles $\hat c$ corresponding to 
the cycles of type (ii) have $t_2= 0$. 
The contribution of these cycles to $W_G$ is then precisely 
the same as the part of the sum with no roots of unity in our definition of $W_\OO$. There is exactly 
one cycle $\hat c$ missing, which has $t_2$ equal to $\pm 1$ depending on whether it is clockwise 
or not. By possibly choosing $\zeta^{-1}$ instead of $\zeta$, we can make the remaining terms in 
$W_G$ and $W_\OO$ be equal, proving the first statement.

	The second statement follows directly from the first and~\cite[Lemma 2.2]{RR85}.
\end{proof}

Let us recall a construction of~\cite{JKS16}. Let $\Pi_n$ be the complete preprojective algebra of type 
$\widetilde{\mathrm A_{n-1}}$, with vertices $1,2,\dots,n$ around the cycle and 
arrows labeled $x_i: (i-1) \to i$ and $y_i: i\to (i-1)$. 
\begin{defin}[\cite{JKS16}]
	Let $B=B(k,n)$ be the quotient of $\Pi_n$ by the closure of the ideal generated by the 
relations $x^k-y^{n-k}$.
\end{defin}

This algebra gives rise to an additive categorification of Scott's cluster algebra structure of the 
coordinate ring of the affine cone over the Grassmannian variety of $k$-spaces in $\mathbb{C}^n$, 
by taking the category $\mathcal F_{k,n}$ of maximal Cohen-Macaulay modules over $B$,~\cite{JKS16}. 
Furthermore, every Postnikov diagram of type $(k,n)$ gives rise to a cluster-tilting object for this 
category,~\cite{BKM16} and from the boundary of its endomorphism algebra we recover the 
algebra $B$:

\begin{prop}[\cite{BKM16}]\label{prop:boundary-alg-P}
	Let $\mathcal P$ be a Grassmannian Postnikov diagram of type $(k,n)$. Then 
$B(\mathcal P)\cong B^{op}$.
\end{prop}

Recall that $n_0=n/d$. 
We define a quotient of $\Pi_{n_0}$ similarly as above. It will give us a basic Morita 
equivalent version of $B*G$. 

\begin{defin}
	Let $B_G = B_G(n_0, k, n)$ be the quotient of $\Pi_{n_0}$ by the ideal generated by the 
relations $x^k-y^{n-k}$.
\end{defin}

The group $G$ acts on $B$ by letting the generator act by the quiver automorphism rotating $i$ to 
$i+\on{GCD}(n,k)$. Denote this automorphism by $g$.

\begin{prop}\label{prop:iso-boundary}
	Let $\tilde{e} = e_1+\cdots +e_{n_0}$ be the idempotent in $B$ corresponding to the first $n_0$ vertices 
of $\Pi_n$. 
Then $\tilde{e}\otimes 1$ is a Morita idempotent in $B*G$, and there is an isomorphism 
$B_G\cong (\tilde{e}\otimes 1)B*G(\tilde{e}\otimes 1)$ mapping $e_i$ to 
$e_i\otimes 1$, $x_i$ and $y_i$ to 
$x_i\otimes 1$ and $y_i\otimes 1$ for $i\neq 1$, $x_1$ to $x_1\otimes g^{-1}$, and $y_1$ to 
$y_{n_0+1}\otimes g$.
\end{prop}

\begin{proof}
	The first assertion follows from~\cite{RR85} since the first $n_0$ 
vertices form a cross-section of vertices of the quiver of $B$ under the action of $G$. 
The explicit isomorphism is a direct application 
of~\cite[Lemma~4.6]{GPP19}, observing that $y_{n_0+1}\otimes g = (1\otimes g)(y_1\otimes 1)$.
\end{proof}

From now on we will freely identify $B_G$ with $(\tilde{e}\otimes 1)B*G(\tilde{e}\otimes 1)$ 
using this isomorphism.

\begin{cor} \label{cor:boundary-alg-orbifold}
	Let $n=dn_0$, let $\OO$ be a Grassmannian orbifold diagram of order $d$ and of 
type $(k,n)$. Then
we have that $B(\OO)\cong (B_G)^{op}$. 
\end{cor}
\begin{proof}
	We have $B(\OO)\sim B(\mathcal P)*G\sim B^{op}*G\sim (B_G)^{op}$, and both algebras are basic. 
\end{proof}

%
\subsection{Modules for the skew group algebra $B_G$}\label{ssec:modules-skew}

For the rest of the paper, we assume that $\OO$ is a Grassmannian orbifold diagram of type $(k,n)$ 
and of order $d$, see Definition~\ref{def:DefOrbiDiagram}. In particular, $k=n_0w_+$ for $w_+$ 
a common winding number of all strands. 
Thus its universal cover $\mathcal P = {\sym}_d(\OO)$ is a $d$-symmetric 
Grassmannian Postnikov diagram of type $(k,n)$. 
Our goal is to explain the relationship between the boundary algebras $B(\mathcal P)$ and $B(\OO)$ 
of the two diagrams. 
By the above results (Proposition~\ref{prop:boundary-alg-P} and 
Corollary~\ref{cor:boundary-alg-orbifold}), 
these algebras are isomorphic to (the opposites of) $B$ and $B_G$ respectively, 
independently of $\OO$ and its symmetrized version, 
so we will focus our attention on the algebras $B$ and $B_G$, 
for which we have a quiver description.

The algebra $B$ and its singularity category have been thoroughly studied, see for instance 
\cite{JKS16}, \cite{DL16}, \cite{BBGE19}. 
We are interested in carrying out a similar study for $B_G$.

The element $t=\sum_ix_iy_i$ is central in $B$ and in fact, $Z(B)= \C[[t]]$, \cite{JKS16}. 
Its image $(\tilde{e}\otimes 1)(t\otimes 1)(\tilde{e}\otimes 1)$ is a central element of $B_G$.  

We are interested in special $B$-modules, namely the rank one Cohen-Macaulay $B$-modules. They 
give rise to cluster-tilting objects in the category $\mathcal F_{k,n}$ of maximal Cohen-Macaulay 
modules over $B$. Furthermore, every object in $\mathcal F_{k,n}$ has a filtration by such modules. 
These modules are constructed as follows.
\begin{defin}\label{def:rk1}
	Let $I$ be a $k$-element subset of $\{1, \ldots, n\}$. Let $L_I$ be the $B$-module given as a representation by: 
	\begin{itemize}
		\item A copy of $Z(B)$ at every vertex. Call $\mathbf 1_i$ the identity of $Z(B)$ at vertex $i$. 
		\item The arrow $x_i$ maps $\mathbf{1}_{i-1}$ to $\mathbf{1}_i$ if $i\in I$, maps $\mathbf{1}_{i-1}$ to $t\mathbf{1}_i$ otherwise.
		\item The arrow $y_i$ maps $\mathbf{1}_i$ to $t\mathbf{1}_i$ if $i\in I$, maps $\mathbf{1}_i$ to $\mathbf{1}_{i-1}$ otherwise.
	\end{itemize}
\end{defin}
\begin{rmk}
	The module $L_I$ is free of rank $n$ over $Z(B)$. It is in fact Cohen-Macaulay of rank one, and all Cohen-Macaulay modules of rank one over $B$ are of this form for some $I$, by~\cite{JKS16}. 
\end{rmk}
\begin{rmk} 
	By construction, we have canonical isomorphisms $\tw{g}{L}_I= L_{I-n_0}$, where $\tw{g}{M}$ denotes the twist of $M$ by $g$ in $\on{mod}(B)$.
\end{rmk}

\begin{ex}
We recall how the modules $L_I$ can be visualised as lattice diagrams (\cite[Section 5]{JKS16}) by 
presenting $L_{127}$ for $\sym_3(\mathcal{O})$ of Example \ref{Ex:CentralFixedPoint} in 
Figure \ref{fig:ExModule127}.  
\end{ex}

\begin{figure}[ht]
	\includegraphics[scale=.65]{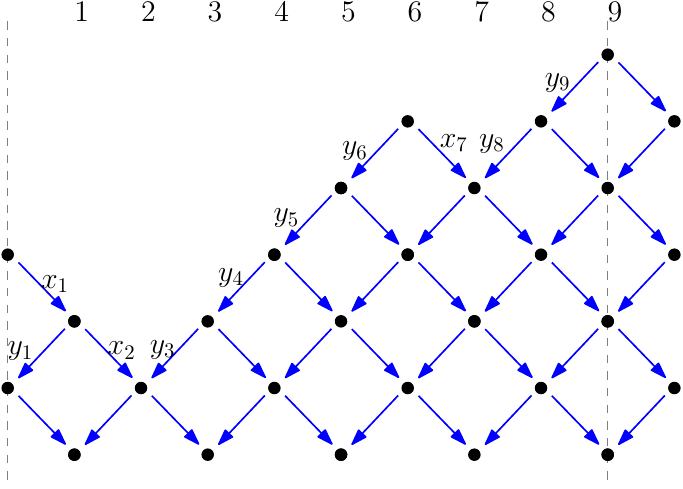} 
	\caption{The lattice diagram for the module $L_{127}$ for $n=9$.}
	\label{fig:ExModule127}
\end{figure}

We want to find analogous modules as the rank 1 modules from Definition~\ref{def:rk1} 
for the algebra $B_G$. 
Let us write $t$ for the element $(\tilde{e}\otimes 1)(t\otimes 1)(\tilde{e}\otimes 1)\in B_G$. 
Let $[I]_{n_0}$ be an equivalence class of $k$-element subsets of $\{1, \dots, n\}$ under the equivalence 
$\sim_{n_0}$ of $2^{\{1, \dots, n\}}$ (these are labels of regions of orbifold diagrams, as introduced in 
Section~\ref{sec:labels}).

\begin{defin}\label{Def:B_GModules}
	Let $I$ be a $k$-subset of $\{1,2,\dots, n\}$ and let $[I]_{n_0}$ be its equivalence class, 
for $n=dn_0$. 
We define a $B_G$-module as follows. As a vector space, we define
\[
L_{ [I]_{n_0}} = \bigoplus_{l=0}^{d-1}\bigoplus_{i=1}^{n_0} \C[[t]], 
\]
where we denote the identities of the above power series rings by $\mathbf{1}_i^l$. 
It is enough to describe the action of the elements $e_i, x_i, y_i\in B_G$ on the elements 
$\mathbf{1}_j^l$. Addition on the superscripts $l$ is always modulo $n_0$. 
	\begin{itemize}
		\item The element $e_i$ maps $\mathbf{1}_i^l$ to $\mathbf{1}_i^l$. 
		\item The arrow $x_1$ maps  $\mathbf{1}_{n_0}^{ l-1}$ to $\mathbf{1}_{1}^{ l}$ if 
		$1+ln_0\in I$, it maps $\mathbf{1}_{n_0}^{l-1}$ to $t\mathbf{1}_{1}^l$ if $1+ln_0\notin I$. 
		\item The arrow $y_1$ maps $\mathbf{1}_{1}^l$ to $\mathbf{1}_{n_0}^{l-1}$ if 
		$1+ln_0\notin I$, 
		it maps $\mathbf{1}_{1}^l$ to $t\mathbf{1}_{n_0}^{l-1}$ if $1+ln_0\in I$
	\end{itemize}
For $i=\{2,\dots, n_0\}$, the $x_i$ and $y_i$ act as follows: 
	\begin{itemize}
		\item 
		The arrow $x_i$, $i=\{2,\dots, n_0\}$, maps $\mathbf{1}_{i-1}^l$ to $\mathbf{1}_{i}^l$ if 
		$i+ln_0\in I$, it maps $\mathbf{1}_{i-1}^l$ to $t\mathbf{1}_{i}^l$ if $i+ln_0\notin I$. 
		\item 
		The arrow $y_i$ maps $\mathbf{1}_{i}^l$ to $\mathbf{1}_{i-1}^l$ if 
		$i+ln_0\notin I$, it maps $\mathbf{1}_{i}^l$ to $t\mathbf{1}_{i-1}^l$ if $i+ln_0\in I$. 
	\end{itemize}
\end{defin}

Observe that the definition of $L_{[I]_{n_0}}$ does indeed depend on the choice of 
$I\in [I]_{n_0}$, see Example~\ref{ex:module-BG} below. 
So this seems not well-defined at first sight. However, we will show in 
Lemma~\ref{lem:iso-representatives} that different choices 
of representatives for $[I]_{n_0}$ result in modules which are canonically isomorphic, 
cf. also Remarks~\ref{rem:iso-choice} and~\ref{rem:choice-I}. 

\begin{ex}\label{ex:module-BG}
We illustrate Definition~\ref{Def:B_GModules} on Example \ref{Ex:CentralFixedPoint}. 
Let $n_0=d=3$ and so $n=9$. The quiver with potential of the above example 
is described in Example \ref{Ex:QPCentralFixedPoint}. 
Recall that $[147]_{3}=\{\{1,4,7\}\}$ because the corresponding alternating region contains the 
orbifold point. All other equivalence classes contain three $3$-subsets of $9$. 
In Figure~\ref{fig:ExModules} we give the modules $L_{[I]_3}$ and $L_{[J]_3}$ for 
$I=\{1,2, 7\}$ and $J=\{1,4,5\}$ as lattice diagrams (similarly as for the rank 1 modules for 
$B$ in~\cite[Section 5]{JKS16}). 
Note that $[127]_3=[145]_3$.  
Figure~\ref{fig:ExModules-layers} gives the top part of the 
three different modules $L_{[127]_3}$, $L_{[145]_3}$ 
and $L_{[478]_3}$ as lattices on the three columns for the three vertices of the algebra 
$B_G$, with the layers $l=0,l=1,l=2$ in different colours. The vertices and arrows outside 
of the middle region are repeated as empty circles and dashed arrows to the left and right. 
\end{ex}

\begin{figure}[ht]
	\includegraphics[scale=.55]{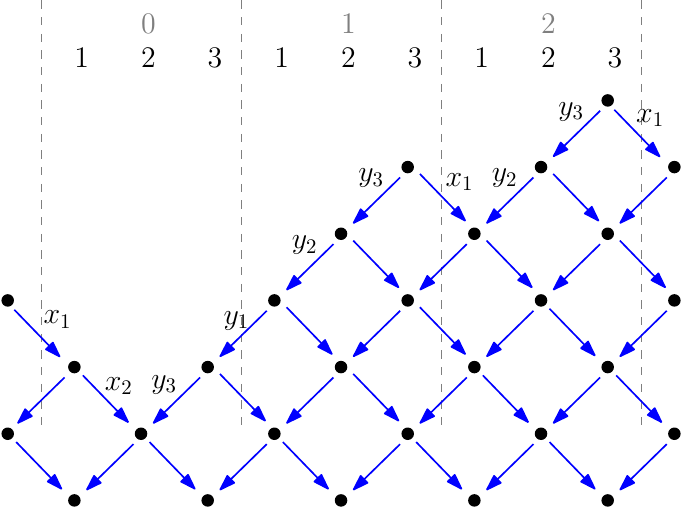}\hspace{1.5cm}
	\includegraphics[scale=.55]{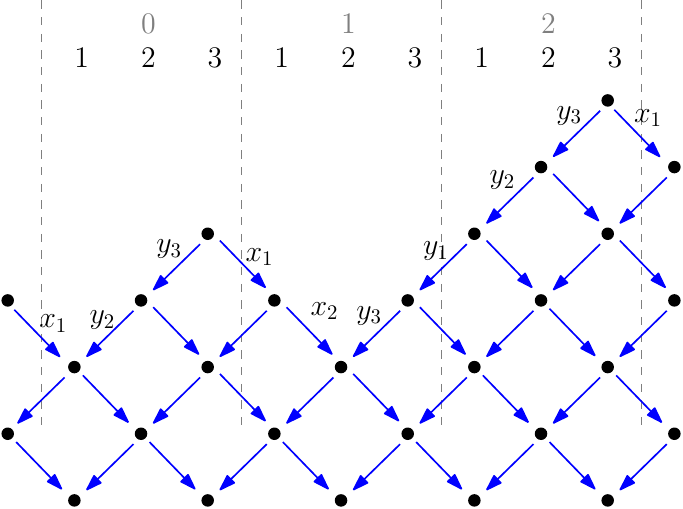}  
	\caption{On the left the module $L_{[127]_3}$ and on the right $L_{[145]_3}$.}
	\label{fig:ExModules}
\end{figure}

\begin{figure}[ht]
	\includegraphics[scale=.8]{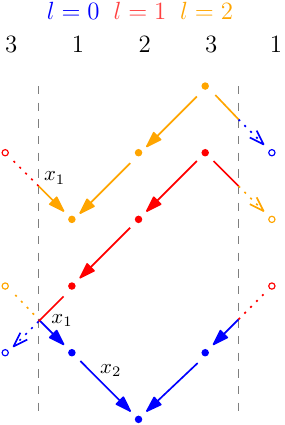}\hspace{.7cm}
	\includegraphics[scale=.8]{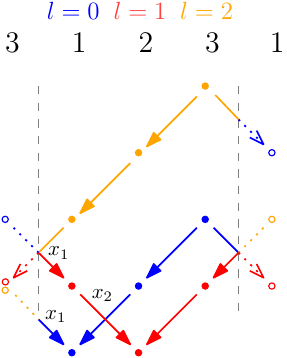}\hspace{.7cm}
	\includegraphics[scale=.8]{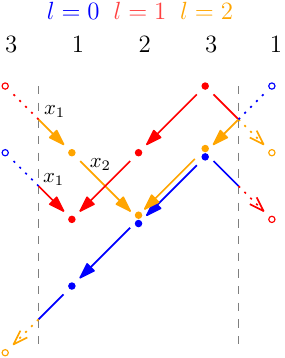}
	\caption{The modules $L_{[127]_3}$, $L_{[145]_3}$ and $L_{[478]_3}$ with 
	their layers coloured.}
	\label{fig:ExModules-layers}
\end{figure}

We define a $\C[[t]]$-linear map $\varphi$ from $L_{[I]_{n_0}}$ to $L_{[I+n_0]_{n_0}}$ 
where $I+n_0$ is 
the $k$-subset obtained from $I$ by adding $n_0$ to each element of $I$. The map 
$\varphi$ increases the label $l$ by 1 modulo $d$, i.e.~we set: 
$\varphi(\mathbf{1}_j^l)=\mathbf{1}_j^{l+1}$ for $l<d-1$ and $\varphi(\mathbf{1}_j^{d-1})=\mathbf{1}_j^{0}$. 

We claim that this map $\varphi$ is a $B_G$-module homomorphism: 

\begin{rmk}\label{Rem:labels}
Let $I$ be a $k$-subset of $\{1,2,\dots, n\}$ and $[I]_{n_0}$ its equivalence class for  $n=dn_0$.  Note that $i+(l+1)n_0\in I+n_0$ if and only if $i+ln_0\in I$.
\end{rmk}

\begin{lemma}\label{lem:iso-representatives}
Let $I$ be a $k$-subset of $\{1,2,\dots, n\}$ and $[I]_{n_0}$ its equivalence class for 
$n=dn_0$. 
The map $\varphi$ induces an isomorphism $L_{[I]_{n_0}}\cong L_{[I+n_0]_{n_0}}$ of 
$B_G$-modules. 
\end{lemma}

\begin{proof}
First, we check $\varphi$ is a homomorphism of $B_G$-modules. 
To ease the notation, let $J=I+n_0$, and so by Remark~\ref{Rem:labels}, 
$i+ln_0\in J$ if and only if $i+(l-1)n_0\in I$ for any $i\in \{1,\dots, n_0\}$, $l=0,\dots, d-1$. 

	We have:
	\begin{itemize}
		
		\item For every $i\in \{1, \dots, n_0\}$,
		\begin{align*}
		e_i \varphi({\mathbf 1}_j^l) 
		&= \begin{cases}
		\mathbf{1}_j^{l+1}& \text{if } i=j;\\
		0 & \text{otherwise}
		\end{cases}\\
		\varphi(e_i {\mathbf 1}_j^l) &= \begin{cases}
		\varphi(\mathbf{1}_j^l)=\mathbf{1}_j^{l+1}  & \text{if } i=j;\\
		\varphi(0)=0 & \text{if $i\ne j$.}
		\end{cases}
		\end{align*}
		
       	\item Since $x_1{\mathbf  1}_j^l=0$ unless $j=n_0$, it is enough 
	to consider  
	the effect of $x_1$ only on $\varphi({\mathbf 1}_{n_0}^{l-1})$: 
		\begin{align*}
		x_1\varphi({\mathbf 1}_{n_0}^{l-1})=x_1\mathbf{1}^{l}_{n_0}
		&= 
		\begin{cases}
		 \mathbf{1}^{l+1}_{1}& \text{if }  1+ln_0\in J\\
		t\mathbf{1}^{l+1}_{1}& \text{if }  1+ln_0\not\in J\\
		\end{cases}\\
		&= 
		\begin{cases}
		 \mathbf{1}^{l+1}_{1}& \text{if }  1+(l-1)n_0\in I\\
		t\mathbf{1}^{l+1}_{1}& \text{if }  1+(l-1)n_0\not\in I\\
		\end{cases}\\
		&= 	\varphi(x_1{\mathbf 1}_{n_0}^{l-1}),
		\end{align*}
		\item  and the effect of $y_1$ on $\varphi({\mathbf 1}_1^{l-1})$:
		\begin{align*}
		y_1\varphi({\mathbf 1}_{1}^{l-1})=y_1\mathbf{1}^{l}_{1}
		&= 
		\begin{cases}
		 \mathbf{1}^{l-1}_{n_0}& \text{if }  1+ln_0\notin J\\
		t\mathbf{1}^{l-1}_{n_0}& \text{if }  1+ln_0\in J\\
		\end{cases}\\
		&= 
		\begin{cases}
		 \mathbf{1}^{l-1}_{n_0}&  \text{ if } i+(l-1)n_0\notin I\\
		t\mathbf{1}^{l-1}_{n_0}&  \text{ if } i+(l-1)n_0\in I\\
		\end{cases}\\
		&= 	\varphi(y_1{\mathbf 1}_1^{l-1}).
		\end{align*}	
		
		\item For $i\in \{2, \dots, n_0\}$,  the effect of $x_i$ on 
		$\varphi({\mathbf 1}_{i-1}^l)$ is:
		\begin{align*}
		x_i\varphi({\mathbf 1}_{i-1}^l)=x_i\mathbf{1}^{l+1}_{i-1}
		&= 
		\begin{cases}
		 \mathbf{1}^{l+1}_{i}& \text{if }   i+(l+1)n_0\in J\\
		t\mathbf{1}^{l+1}_{i}& \text{if }  i+(l+1)n_0\not\in J\\
		\end{cases}\\
		&= 
		\begin{cases}
		 \mathbf{1}^{l+1}_{i}& \text{if }  i+ln_0\in I\\
		t\mathbf{1}^{l+1}_{i}& \text{if }  i+ln_0\not\in I\\
		\end{cases}\\
		&= 	\varphi(x_i{\mathbf 1}_{i-1}^l).
		\end{align*}
		\item For $i\in \{2, \dots, n_0\}$,  the effect of $y_i$ on 
		$\varphi({\mathbf 1}_{i}^l)$ is 
		\begin{align*}
		y_i\varphi({\mathbf 1}_{i}^l)=y_i\mathbf{1}^{l+1}_{i}
		&= 
		\begin{cases}
		 \mathbf{1}^{l+1}_{i-1}& \text{if }   i+(l+1)n_0\notin J\\
		t\mathbf{1}^{l+1}_{i-1}& \text{if }   i+(l+1)n_0\in J\\
		\end{cases}\\
		&= 
		\begin{cases}
		 \mathbf{1}^{l+1}_{i-1}& \text{if }   i+ln_0\notin I\\
		t\mathbf{1}^{l+1}_{i-1}& \text{if }  i+ln_0\in I\\
		\end{cases}\\
		&= 	\varphi(y_i{\mathbf 1}_i^l).
		\end{align*}		
	\end{itemize}

For the bijectivity, we note that $\varphi$ permutes the generators of the modules 
and that the generators freely generate the modules over the centre. 
\end{proof}

\begin{rmk}\label{rem:iso-choice}
By Lemma~\ref{lem:iso-representatives}, the module $L_{[I]_{n_0}}$ is well defined. 
\end{rmk}

\begin{rmk}
	As we have mentioned, from the definition it is clear that $x_iy_i$ maps $\mathbf{1}_i^l$ 
to $t\mathbf{1}_i^l$, so calling all the variables in the power series rings in 
Definition~\ref{Def:B_GModules} $t$ is justified. 
\end{rmk}

We want to relate the $B_G$-modules $L_{[I]_{n_0}}$ to the $B$-modules $L_I$. For this, we need to introduce 
some notation. 
The map $B\to B*G$ given by $b\mapsto b\otimes 1$ induces a functor $F = (B*G)\otimes_B -$ from 
$\on{mod}(B)$ to $\on{mod}(B*G)$. There is an equivalence 
$j^*:\on{mod}(B*G)\to \on{mod}(B_G)$ given by 
\[
j^* = (\tilde{e}\otimes 1)B*G\otimes_{B*G}-
\]
using the isomorphism of Proposition~\ref{prop:iso-boundary} where $\tilde{e}=e_1+\cdots + e_{n_0}$ 
is the idempotent of the first $n_0$ vertices of $\Pi_n$. 

We aim to prove that $L_{[I]_{n_0}}\cong j^*F(L_I)$. Let us do some preparation. 
As a $B$-module, $(B*G)\otimes_B L_I$ is generated by the elements $(1\otimes g^l)\otimes_B \mathbf 1_h$. This in turn implies that the $B_G$-module $j^*F(L_I)$ is generated by elements of the form $(e_i\otimes g^l)\otimes_B \mathbf{1}_h$. However, these elements are nonzero (if and) only if $h= g^{-l}(i)$. We are left with considering the elements $(e_i\otimes g^l)\otimes_B \mathbf{1}_{g^{-l}(i)}$, for $i\in \{1, \dots, n_0\}$ and $l\in \{0, \dots, d-1\}$. Observe that $t$ (which again we use as notation for $(\tilde{e}\otimes 1)(t\otimes 1)(\tilde{e}\otimes 1)$ as well as for $\sum_{i=1}^{n_0}x_iy_i$ in the quiver description of $B_G$) is in the center of $B_G$, so that as a $\C[[t]]$-module we have a decomposition $j^*F(L_I) = \bigoplus_{l=0}^{d-1}\bigoplus_{i=1}^{n_0} \C[[t]]$. It is therefore natural to define a map of $\C[[t]]$-modules $\psi:  L_{[I]_{n_0}}\to j^*F(L_I)$ by setting 
\[ 
\psi( \mathbf{1}_i^l) = (e_i\otimes g^{-l})\otimes_B \mathbf{1}_{g^{l}(i)}.
\]

\begin{lemma}\label{lem:L_I}
	The map $\psi:L_{[I]_{n_0}}\to  j^*F(L_I) $ is an isomorphism of $B_G$-modules.
\end{lemma}
\begin{proof}
The strategy is the same as in the proof of Lemma~\ref{lem:iso-representatives}. 
First of all, $\psi$ permutes the free generators of the corresponding modules, giving 
the bijection. 

To see that $\psi$ is a $B_G$-module homomorphism, we check that $e_i, x_i, y_i$ 
in $B_G$ act on $\mathbf{1}_j^l$ in the same way as the corresponding elements in 
$(\tilde{e}\otimes 1)(B*G)(\tilde{e}\otimes 1)$ act on $(e_j\otimes g^{-l})\otimes_B \mathbf{1}_{g^{l}(j)}$ 
by left multiplication. Note that $g^l(j)=j+ln_0$. 

As before, for the action of $x_i$, we will restrict to $j=i-1$, 
for the action of $y_i$ to $j=i$. 
%
	
	We have:
	\begin{itemize}
		
		\item For every $i\in \{1, \dots, n_0\}$, 
		\begin{align*}
		(e_i\otimes 1)\psi( \mathbf{1}_j^l)&= 	
		(e_i\otimes 1)(e_j\otimes g^{-l})\otimes_B \mathbf{1}_{g^{l}(j)}\\
		& = (e_ie_j\otimes g^{-l})\otimes_B \mathbf{1}_{g^l(j)} \\
		&= \begin{cases}
		(e_i\otimes g^{-l})\otimes_B \mathbf{1}_{g^{l}(i)}& \text{if } i=j\\
		0 & \text{otherwise}
		\end{cases}\\
		&= \begin{cases}
		\psi( \mathbf{1}_i^l) & \text{if } i=j;\\
		0 & \text{otherwise}
		\end{cases}\\
		&= \psi(e_i  \mathbf{1}_j^l).
		\end{align*}
		
		\item For $i\in \{2, \dots, n_0\}$ (and $j=i-1$), 
		\begin{align*}
		(x_i\otimes 1)\psi( \mathbf{1}_{i-1}^l) &= 
		(x_i\otimes 1)(e_{i-1}\otimes g^{-l})\otimes_B \mathbf{1}_{g^{l}(i-1)} \\
		&= 	(x_i\otimes g^{-l})\otimes_B \mathbf{1}_{(i-1)+ln_0}  \\
		&= (e_i\otimes g^{-l})(g^{l}(x_i)\otimes 1)\otimes_B \mathbf{1}_{(i-1)+ln_0} 
		  \mbox{(using Definition~\ref{def:skew-group})}\\
		& \mbox{now $g^l(x_i)\otimes 1=x_{i+ln_0}\otimes 1$ 
		is in $B$ and we can pull it across $\otimes_B$ 
		to get} \\
		&=  (e_i\otimes g^{-l})\otimes_B
		\begin{cases}
		\mathbf{1}_{i+ln_0} & \text{if }  i+ln_0\in I;\\
		t \mathbf{1}_{i+ln_0} & \text{if } i+ln_0 \not\in I 
		\end{cases}\\
		&= 
		\begin{cases}
		\psi(\mathbf{1}_i^l) & \text{if }  i+ln_0\in I;\\
		\psi(t \mathbf{1}_i^l) & \text{if } i+ln_0 \not\in I 
		\end{cases}\\
		&= 	\psi(x_{i} \mathbf{1}_{i-1}^l).
		\end{align*}
		\item For $i\in \{2, \dots, n_0\}$ (and $j=i$), 
		\begin{align*}
		(y_i\otimes 1)\psi( \mathbf{1}_i^l) &= 
		(y_i\otimes 1)(e_i\otimes g^{-l})\otimes_B \mathbf{1}_{g^{l}(i)} 
		= (y_{i}\otimes g^{-l}) \otimes_B \mathbf{1}_{i+ln_0} \\
		&= (e_{i-1}\otimes g^{-l})(g^l(y_i)\otimes 1)\otimes_B\mathbf{1}_{i+ln_0} \\
		& = (e_{i-1}\otimes g^{-l})\otimes_B 
		\begin{cases}
		\mathbf{1}_{(i-1)+ln_0}& \text{if }  i+ln_0 \not\in I\\
		t\mathbf{1}_{(i-1)+ln_0}& \text{if } i+ln_0 \in I\\
		\end{cases}\\
		&= 
		\begin{cases}
		\psi( \mathbf{1}_{i-1}^l) & \text{if }   i+ln_0\not\in I\\
		\psi(t \mathbf{1}_{i-1}^l) & \text{if }   i+ln_0 \in I \\
		\end{cases}\\
		&= 	\psi(y_i \mathbf{1}_i^l).
		\end{align*}
		\item Finally, for $i=1$ and $j=n_0$, recalling that $x_1\in B_G$ maps to 
		$x_1\otimes g^{-1}\in (\tilde{e}\otimes 1)B*G(\tilde{e}\otimes 1)$ via the isomorphism of Proposition~\ref{prop:iso-boundary},  
		\begin{align*}
		(x_1\otimes g^{-1})\psi( \mathbf{1}_{n_0}^l) &= 
		(x_1\otimes g^{-1})(e_{n_0}\otimes g^{-l})\otimes_B \mathbf{1}_{g^{l}(n_0)} 
		= x_1\otimes g^{-l-1}\otimes_B \mathbf{1}_{(l+1)n_0} \\
		& = (e_1\otimes g^{-l-1})(g^{l+1}(x_1)\otimes 1)\otimes_B\mathbf{1}_{(l+1)n_0}  \\
		&= (e_1\otimes g^{-l-1})\otimes_B x_{1+(l+1)n_0} \mathbf{1}_{(l+1)n_0} \\
		&= (e_1\otimes g^{-l-1})\otimes_B 
		\begin{cases}
		\mathbf{1}_{g^{l+1}(1)}& \text{if } 1+(l+1)n_0\in I\\
		t\mathbf{1}_{g^{l+1}(1)}& \text{if } 1+(l+1)n_0\not\in I 
		\end{cases}\\
		&= \begin{cases}
		\psi( \mathbf{1}_1^{l+1}) & \text{if } 1+(l+1)n_0\in I\\
		\psi(t \mathbf{1}_1^{l+1}) & \text{if } 1+(l+1)ln_0 \not\in I
		\end{cases}\\
		&= 
		\psi(x_1 \mathbf{1}_{n_0}^l)
		\end{align*}
		\item and to check $y_1$, we only consider $j=1$, recalling that 
		$y_1$ maps to $y_{n_0+1}\otimes g\in (\tilde{e}\otimes 1)B*G(\tilde{e}\otimes 1)$ via the isomorphism of Proposition~\ref{prop:iso-boundary},
		\begin{align*}
		(y_{n_0+1}\otimes g)\psi( \mathbf{1}_1^l) &= 
		(y_{n_0+1}\otimes g)(e_1\otimes g^{-l})\otimes_B \mathbf{1}_{g^{l}(1)} 
		=( y_{n_0+1}\otimes g^{-l+1})\otimes_B \mathbf{1}_{1+ln_0}\\
		&= (e_{n_0}\otimes g^{-l+1})(g^{l-1}(y_{n_0+1})\otimes 1)\otimes_B \mathbf{1}_{1+ln_0}  \\ 
		& = (e_{n_0}\otimes g^{-l+1}) \otimes_B y_{1+ln_0} \mathbf{1}_{1+ln_0} \\
		&= (e_{n_0}\otimes g^{-l+1}) \otimes_B 
		\begin{cases}
		\mathbf{1}_{g^{l-1}(n_0)}& \text{if } 1+ln_0 \not\in I\\
		t\mathbf{1}_{g^{l-1}(n_0)}& \text{if } 1+ln_0\in I
		\end{cases}\\
		&= \begin{cases}
		\psi( \mathbf{1}_{n_0}^{l-1}) & \text{if } 1+ln_0 \not\in I\\
		\psi(t \mathbf{1}_{n_0}^{l-1}) & \text{if }  1+ln_0 \in I\\
		\end{cases}\\
		&= 
		\psi(y_1 \mathbf{1}_1^l).
		\end{align*}		
	\end{itemize}
\end{proof}

\begin{rmk}\label{rem:choice-I}
	Since $F(M) = F(\tw{g}{M})$ for any $M\in \on{mod}(B)$, we recover 
	that the modules $L_{[I]_{n_0}}$ are well defined (cf. Remark~\ref{rem:iso-choice}). 
\end{rmk}	

\begin{lemma}\label{lem:basic}
	If $[I]_{n_0}\neq [J]_{n_0}$ then $L_{[I]_{n_0}}\not\cong L_{[J]_{n_0}}$.
\end{lemma}

\begin{proof}
	Assume that $L_{[I]_{n_0}}\cong L_{[J]_{n_0}}$. By Lemma~\ref{lem:L_I}, we have that $j^*F(L_I)\cong j^*F(L_J)$ hence $F(L_I)\cong F(L_J)$. This implies that $L_I\cong \tw{g}{L}_J$ for some $g\in G$. We conclude that $I$ and $J$ differ by a multiple of $n_0$ and so $[I]_{n_0}= [J]_{n_0}$.
\end{proof}

\begin{cor}
Let $I$ and $J$ be $k$-subsets of $n=dn_0$. Then $L_{[I]_{n_0}}\cong L_{[J]_{n_0}}$ if and only if 
$[I]_{n_0}= [J]_{n_0}$. 
\end{cor}

%
\subsection{The algebra $A(\mathcal O)$ as endomorphism algebra}

Let now $T_\mathcal P$ be the $B$-module defined by 
\[
T_\mathcal P= \bigoplus_{I\in \mathcal I} L_I.
\]
Since $\mathcal P$ is $d$-symmetric, it is invariant under rotation by $n_0$ steps. It follows 
that $T_\mathcal P= \tw{g}{T}_\mathcal P$, since $L_{I-n_0} = \tw{g}{L}_I$. 
As the labels 
correspond to regions on the disk, they either come in orbits of length $d$ 
(i.e.~are acted upon freely by $G$) or are fixed, and there can be at most one fixed label 
(the label of the central region, if it is alternating). 
\begin{defin}
Let $\OO$ be an orbifold diagram on a disk with $n_0$ points. 
	Let $T_\OO$ be the $B_G$-module defined by 
	\[
	T_\OO = \bigoplus_{[I]_{n_0} \in \mathcal I_\OO} L_{[I]_{n_0}}.
	\]
\end{defin}

We need some more notation. Let $T_0 = L_I$ for the label $I$ of the central region of $\mathcal P$, if it is alternating, and $T_0=0$ otherwise. The group $G$ acts freely on $T_\mathcal P\setminus T_0$, and we call $T_\mathcal P'$ a chosen cross-section of this action. Finally, we call $T_\mathcal P^{red} = T_0\oplus T_\mathcal P'$.

\begin{lemma}\label{lem:T}
	We have $T_\OO = j^*F(T_\mathcal P^{red})$.
\end{lemma}

\begin{proof}
	Use Lemma~\ref{lem:L_I} and Definition~\ref{defin:I_O}.
\end{proof}

\begin{thm}
	\label{thm:main}
	With the above notation, we have $A(\OO)\cong \on{End}_{B_G}(T_\OO)$.
\end{thm}

\begin{proof}
	By Proposition~\ref{prop:sgas}, we know that $A(\OO)$ is Morita equivalent to $A(\mathcal P)*G$, where the action of a generator is given by rotating $n_0=\on{GCD}(k,n)$ steps clockwise. 
We recall the isomorphism $A(\mathcal P)\cong \on{End}_B(T_\mathcal P)$ from~\cite[Section 10]{BKM16}. 
This isomorphism 
arises from sending every arrow $\alpha:I\to J$ in $Q_{\mathcal P}$ with $I,J$ vertices of $\mathcal P$ to the 
(injective) minimal codimension map $L_I\to L_J$ (sending the lattice diagram of the module $L_I$ as high up 
as possible into the lattice diagram of $L_J$). As $\mathcal P$ is $d$-symmetric, 
if $\alpha:I\to J$ is an arrow between two vertices which do not correspond to the central region, it appears with 
$d-1$ ``rotated'' copies: there are arrows $\alpha_m:I+mn_0 \to J+mn_0$ for $m=0,\dots, d-1$ where  
$\alpha_0=\alpha$. 
Similarly, if $I$ corresponds to the central region, there are arrows $\alpha_m:I\to J+mn_0$ for $m=0,\dots, d-1$ or 
if $J$ is at the central region, there are arrows $\alpha_m:I+mn_0 \to J$ for $m=0,\dots, d-1$.
With the action by twists on $\on{End}_B(T_\mathcal P)$ as in Section\ref{ssec:modules-skew}, 
the isomorphism $A(\mathcal P)\cong \on{End}_B(T_\mathcal P)$ is $G$-invariant. 
We get then an isomorphism $A(\mathcal P)*G\cong \on{End}_B(T_\mathcal P)*G$. 
Now we can apply Lemma~\ref{lm:old-lemma-5-3} and Lemma~\ref{lm:morita-equ} 
to $M_0\oplus M=T_\mathcal P=T_0\oplus F(T'_\mathcal P)$, 
so $M_1=T'_\mathcal P$.


We obtain a Morita equivalence
	\[
	\on{End}_B(T_\mathcal P)*G\sim \on{End}_{\on{mod}(B)^G}\left(F(T_0)\oplus F(T'_\mathcal P)\right) = \on{End}_{\on{mod}(B)^G}\left(F(T_\mathcal P^{red})\right)
	\]
		
	The latter is in turn Morita equivalent to $\on{End}_{B*G}(F(T_\mathcal P^{red}))$ and then to 
	$\on{End}_{B_G}(j^*F(T_\mathcal P^{red}))$, since both $E$ and $j^*$ are equivalences.
	Finally, by Lemma~\ref{lem:T}, the latter equals $\on{End}_{B_G}(T_\OO)$.
	We have proved that $A(\OO)\sim \on{End}_{B_G}(T_\OO)$. The statement follows since both algebras are basic (the latter by Lemma~\ref{lem:basic}).
\end{proof}

\begin{rmk}\label{rem:concluding}
	We conclude with a remark motivated by the following question: in~\cite{BKM16}, 
the dimer algebra $A$ is shown to be isomorphic to the endomorphism algebra of a module 
$T$, which is a cluster tilting object in a Frobenius, stably 2-Calabi-Yau category. Is the same 
true in our case?
	By results of Demonet (\cite[\S 2.2.4]{Demonet-2011}), it is indeed the case that the skew 
group category $\on{CM}(B)*G$ of the category of Cohen-Macaulay $B$-modules is 
Frobenius and stably 2-CY, and our module $T_\OO$ does lie in it. Moreover, $F$ maps 
$G$-invariant cluster tilting objects to cluster tilting objects, so indeed $T_\OO$ is cluster tilting. 
	We note however that we do not have a direct description of the category 
$\on{CM}(B)*G$ as (equivalent to) a subcategory of $\on{mod}(B_G)$.
\end{rmk}

\bibliographystyle{alpha}
\bibliography{biblio}

\end{document}